\documentclass{amsart}

\newcommand{\mA}{\mathcal{A}}
\newcommand{\mB}{\mathcal{B}}

\usepackage{amsmath}
\usepackage{amsthm}
\usepackage{graphicx}
\usepackage{color}
\usepackage{amsfonts}
\usepackage{amssymb}
\usepackage{mathrsfs}
\usepackage{amscd}
\usepackage{url}

\newcommand{\re}{\mathbb{R}}
\newcommand{\cpx}{\mathbb{C}}

\newcommand{\F}{\mathbb{F}}
\newcommand{\N}{\mathbb{N}}

\newcommand{\half}{\frac{1}{2}}
\newcommand{\lmd}{\lambda}

\newcommand{\eps}{\epsilon}

\newcommand{\dt}{\delta}

\def\af{\alpha}
\def\bt{\beta}

\newcommand{\sig}{\sigma}
\newcommand{\Sig}{\Sigma}

\newcommand{\mt}[1]{\mathtt{#1}}
\newcommand{\A}{\mathcal{A}}

\newcommand{\reff}[1]{(\ref{#1})}

\newcommand{\mbf}[1]{\mathbf{#1}}

\newcommand{\mc}[1]{\mathcal{#1}}

\newcommand{\bdes}{\begin{description}}
\newcommand{\edes}{\end{description}}

\newcommand{\bal}{\begin{align}}
\newcommand{\eal}{\end{align}}

\newcommand{\bnum}{\begin{enumerate}}
\newcommand{\enum}{\end{enumerate}}

\newcommand{\bit}{\begin{itemize}}
\newcommand{\eit}{\end{itemize}}

\newcommand{\bea}{\begin{eqnarray}}
\newcommand{\eea}{\end{eqnarray}}
\newcommand{\be}{\begin{equation}}
\newcommand{\ee}{\end{equation}}

\newcommand{\baray}{\begin{array}}
\newcommand{\earay}{\end{array}}

\newcommand{\bsry}{\begin{subarray}}
\newcommand{\esry}{\end{subarray}}

\newcommand{\bca}{\begin{cases}}
\newcommand{\eca}{\end{cases}}

\newcommand{\bcen}{\begin{center}}
\newcommand{\ecen}{\end{center}}

\newcommand{\bbm}{\begin{bmatrix}}
\newcommand{\ebm}{\end{bmatrix}}

\newcommand{\bmx}{\begin{matrix}}
\newcommand{\emx}{\end{matrix}}

\newcommand{\bpm}{\begin{pmatrix}}
\newcommand{\epm}{\end{pmatrix}}

\newcommand{\btab}{\begin{tabular}}
\newcommand{\etab}{\end{tabular}}

\newcommand{\thmlist}{
\begin{list}{Step 1}
{\setlength{\leftmargin}{0.6 in}\setlength{\labelwidth}{0.5 in}} }
\newcommand{\alglist}{
\begin{list}{Step 1}
{\setlength{\leftmargin}{1.1 in}\setlength{\labelwidth}{1.0 in}} }
\newtheorem{theorem}{Theorem}[section]

\newtheorem{prop}[theorem]{Proposition}

\newtheorem{lemma}[theorem]{Lemma}

\theoremstyle{definition}

\newtheorem{exm}[theorem]{Example}
\newtheorem{alg}[theorem]{Algorithm}

\numberwithin{equation}{section}

\begin{document}

\title{Symmetric Tensor Nuclear Norms}

\author{Jiawang Nie}
\address{Department of Mathematics,
University of California at San Diego,
9500 Gilman Drive, La Jolla, CA, USA, 92093.}
\email{njw@math.ucsd.edu}



\begin{abstract}
This paper studies nuclear norms of symmetric tensors.
As recently shown by Friedland and Lim,
the nuclear norm of a symmetric tensor
can be achieved at a symmetric decomposition.
We discuss how to compute symmetric tensor nuclear norms, depending on
the tensor order and the ground field.
Lasserre relaxations are proposed for the computation.
The theoretical properties of the relaxations are studied.
For symmetric tensors, we can compute their nuclear norms,
as well as the nuclear decompositions.
The proposed methods can be extended to nonsymmetric tensors.
\end{abstract}

\maketitle

\section{Introduction}

Let $\F$ be a field (either the real field $\re$ or the complex one $\cpx$).
Let $\F^{n_1 \times \cdots \times n_m}$ be the space of tensors
of order $m$ and dimension $(n_1, \ldots, n_m)$.
Each tensor in $\F^{n_1 \times \cdots \times n_m}$
can be represented by a $m$-dimensional hypermatrix (or array)
\[
\mA = ( \mA_{i_1 \ldots i_m} )
\]
with each entry $\mA_{i_1 \ldots i_m} \in \F$ and
$1 \leq i_1 \leq n_1, \ldots, 1 \leq i_m \leq n_m$.
For two tensors $\mA, \mB \in \F^{n_1 \times \cdots \times n_m}$,
their {\it hermitian inner product} is defined as
\be \label{inn:<A,B>}
\mA \bullet \mB  :=
\sum_{1 \leq i_j \leq n_j, j=1,\ldots,m }
\mA_{i_1 \ldots i_m} \bar{\mB}_{i_1 \ldots i_m}.
\ee
(The bar $\bar{\empty}$ denotes the complex conjugate.)
This induces the {\it Hilbert-Schmidt norm}
\be \label{HSnm:||A||}
\| \mA \| \, := \, \sqrt{   \mA \bullet \mA  }.
\ee
For vectors $x^{(1)} \in \F^{n_1}$, $\ldots$,  $x^{(m)} \in \F^{n_m}$,
$x^{(1)} \otimes \cdots \otimes x^{(m)}$
denotes their standard tensor product, i.e.,
\[
(x^{(1)} \otimes \cdots \otimes x^{(m)})_{i_1 \ldots i_m} =
(x^{(1)})_{i_1} \cdots  (x^{(m)})_{i_m}.
\]
The {\it spectral norm} of $\mA$, depending on the field $\F$,
is defined as
\be \label{spc||A||:nsm}
\| \mA \|_{\sig,\F} := \max \{
| \mA \bullet x^{(1)} \otimes \cdots \otimes x^{(m)} | : \,
\| x^{(j)} \|  =1, x^{(j)} \in \F^{n_j}  \}.
\ee
In the above, $\| \cdot \|$ denotes the standard Euclidean vector norm.
The {\it nuclear norm} of $\mA$, also depending on $\F$, is defined as
\be \label{nuc||A||:nsy}
\| \mA \|_{\ast,\F}  := \min \left\{ \sum_{i=1}^r |\lmd_i|
\left| \baray{c}
\mA = \Sig_{i=1}^r  \lmd_i v^{(i,1)} \otimes \cdots \otimes v^{(i,m)}, \\
\| v^{(i,j)} \| = 1, v^{(i,j)} \in \F^{n_j}
\earay\right.
\right \}.
\ee
The spectral norm $\| \cdot \|_{\sig,\F}$
is dual to the nuclear norm $\| \cdot \|_{\ast,\F}$ (cf.~\cite{FriLim14b}):
\[
\| \mA \|_{\sig,\F} = \max \{
| \mA \bullet \mc{X}  |: \, \| \mc{X} \|_{\ast,\F}  = 1 \},
\]
\[
\| \mA \|_{\ast,\F} = \max \{
|  \mA \bullet \mc{Y} |: \, \| \mc{Y} \|_{\sig,\F}  = 1 \}.
\]
Spectral and nuclear tensor norms have important applications,
e.g., signal processing and blind identification (\cite{LimCom10,LimCom14}),
tensor completion and recovery (\cite{MHWG,YuaZha15}),
low rank tensor approximations (\cite{FriOtt13,NW14,ZLQ12}).
%
%
When the order $m>2$, the computation of spectral and nuclear norms is NP-hard
(\cite{FriLim14b,FriLim16a,HL13}).
In \cite{Der13}, the nuclear norms of several interesting tensors
were studied. We refer to \cite{Land12,Lim13} for tensor theory and applications.

This paper focuses on nuclear norms of symmetric tensors.
Recall that a tensor $\mA \in \F^{n_1 \times \cdots \times n_m}$
is symmetric if $n_1 = \cdots = n_m$ and
\[
\mA_{i_1 \ldots i_m} = \mA_{j_1 \ldots j_m}
\]
whenever $(i_1, \ldots, i_m)$ is a permutation of $(j_1, \ldots, j_m)$.
Let $\mt{S}^m( \F^n)$ be the space of all $n$-dimensional symmetric tensors
of order $m$ and over the field $\F$.
For convenience, denote the symmetric tensor power
\[
x^{\otimes m} \, := \, x \otimes \cdots \otimes x \quad
(\mbox{$x$ is repeated $m$ times}).
\]
For a symmetric tensor $\A \in \mt{S}^m( \F^n)$,
its spectral and nuclear norms can be simplified as
(for $\F=\re$ or $\cpx$)
\be \label{||A||sig:sym}
\| \mA \|_{\sig,\F} = \max \{
| \mA \bullet x^{\otimes m} | : \,
\| x  \|  =1, x \in \F^n \},
\ee
\be \label{ast||A||:sym}
\| \mA \|_{\ast,\F} = \min \left\{ \sum_{i=1}^r |\lmd_i|  : \,
\mA = \sum_{i=1}^r  \lmd_i (v_i)^{\otimes m},
\| v_i \| = 1, v_i \in \F^n, \lmd_i \in \F \right \}.
\ee
The equality \reff{||A||sig:sym} can be found in
Banach \cite{banach}, Friedland \cite{Fri13},
Friedland and Ottaviani \cite{FriOtt13}, and Zhang et al. \cite{ZLQ12}.
The equality~\reff{ast||A||:sym} was recently proved
by Friedland and Lim \cite{FriLim14b}.
In \reff{ast||A||:sym}, the decomposition of $\mA$, for which
the minimum is achieved, is called a
{\it nuclear decomposition} as in \cite{FriLim14b}.
When $\mA$ is a real tensor,
\[
\| \mA \|_{\sig,\re} \leq \| \mA \|_{\sig,\cpx}, \qquad
\| \mA \|_{\ast,\re} \geq \| \mA \|_{\ast,\cpx}.
\]
The strict inequalities are possible in the above.
Explicit examples can be found in \cite{FriLim14b}
and in \S\ref{sc:num} of this paper.

The computation of tensor nuclear norms can be formulated
as a moment optimization problem.
When $\mA$ is a real cubic symmetric tensor (i.e., $m=3$),
Tang and Shah \cite{TanSha15} pointed out that
the real nuclear norm $\| \mA \|_{\ast,\re}$ is equal to
the optimal value of the moment optimization problem
\be \label{muopt:A:m=3}
\min  \quad \int_S 1 \mt{d} \mu \quad
s.t. \quad  \mc{A} = \int_S x \otimes x \otimes x \mt{d} \mu
\ee
where $\mu$ is a Borel measure variable whose support is contained in the unit sphere
\be \label{usph:S}
S \, := \, \{ x\in \re^n \mid \, \| x \| = 1 \}.
\ee
The equality constraint in \reff{muopt:A:m=3} gives cubic moments of $\mu$,
while the objective is the total mass of $\mu$.
Lasserre's hierarchy of semidefinite relaxations \cite{Lasserre01,Lasserre08}
can be applied to solve \reff{muopt:A:m=3}, as proposed in \cite{TanSha15}.
This gives a sequence of lower bounds, say, $\{ \rho_k \}$,
for the real nuclear norm $\| \mA \|_{\ast, \re}$.
It can be shown that
$\rho_k \to \| \mA \|_{\ast, \re}$ as $k \to \infty$.
However, in computational practice, it is very difficult to check the convergence,
i.e., how do we detect if $\rho_k$ is equal to, or close to, $\| \mA \|_{\ast, \re}$?
When the convergence occurs, how can we get a nuclear decomposition?
To the best of the author's knowledge,
there was few prior work on these two questions.
The major difficulty is that the flat extension condition
(cf.~\cite{CurtoF,Fialkow,Helton}),
which is often used for solving moment problems,
is usually not satisfied for solving \reff{muopt:A:m=3}.
This causes the embarrassing fact that the nuclear norm
is often not known, although it can be approximated as close as possible in theory.
Moreover, when the order $m$ is even, or the field $\F=\cpx$,
the nuclear norm $\| \mA \|_{\ast, \F}$ is no longer equal to
the optimal value of \reff{muopt:A:m=3}.

In this paper, we propose methods for computing
nuclear norms of symmetric tensors, for both odd and even orders,
over both the real and complex fields.
We give detailed theoretical analysis and computational implementation.
\bit

\item When the order $m$ is odd and $\F = \re$,
the nuclear norm $\| \mA \|_{\ast, \re}$ equals the optimal value of \reff{muopt:A:m=3},
as shown in \cite{TanSha15}.

\item When the order $m$ is even and $\F = \re$, the nuclear norm
$\| \mA \|_{\ast, \re}$ is no longer equal to the optimal value of
\reff{muopt:A:m=3}. We construct a new moment optimization problem
whose optimal value equals $\| \mA \|_{\ast, \re}$.

\item When $\F = \cpx$, we construct a new moment optimization problem
whose optimal value equals $\| \mA \|_{\ast, \cpx}$,
for both even and odd orders.

\eit
Lasserre relaxations in \cite{Lasserre01,Lasserre08}
are efficient for solving these moment optimization problems.
We can get a sequence of lower bounds for the nuclear norm $\| \mA \|_{\ast, \F}$,
which is deonoted as $\{ \| \mA \|_{k\ast, \F} \}_{k=1}^{\infty}$.
(The integer $k$ is called the relaxation order.)
We prove the asymptotic convergence
$\| \mA \|_{k\ast, \F} \to \| \mA \|_{\ast, \F}$
as the relaxation order $k \to \infty$.
In computational practice, the finite convergence often occurs,
i.e., $\| \mA \|_{k\ast, \F} = \| \mA \|_{\ast, \F}$ for some $k$.
We show how to detect $\| \mA \|_{k\ast, \F} = \| \mA \|_{\ast, \F}$
and how to compute nuclear decompositions.
This can be done by solving a truncated moment problem.
We also prove conditions that guarantee finite convergence.

The paper is organized as follows.
Section~\ref{sc:Real:oddm} discusses nuclear norms
when the order $m$ is odd and $\F = \re$.
Section~\ref{sc:Real:meven} discusses nuclear norms
when $m$ is even and $\F = \re$.
Section~\ref{sc:cpx} discusses nuclear norms
when the field $\F = \cpx$.
The numerical experiments are given in Section~\ref{sc:num}.
Some preliminary results are given in
Section~\ref{sc:prelim}.
The extensions to nonsymmetric tensors
are given in Section~\ref{sc:exten}.

\section{Preliminaries}
\label{sc:prelim}

\subsection*{Notation}\,
The symbol $\N$ (resp., $\re$, $\cpx$) denotes the set of
nonnegative integers (resp., real, complex numbers).
%
%
For $x=(x_1,\ldots,x_n)$ and $\af = (\af_1, \ldots, \af_n) \in \N^n$,
denote
\[
x^\af := x_1^{\af_1}\cdots x_n^{\af_n}, \quad
|\af| := |\af_1| + \cdots + |\af_n|.
\]
For a degree $d>0$, denote the set of monomial powers
\be \label{N[0d]}
\left\{\baray{rcl}
\N^n_{[0,d]} &:=& \{ \af \in \N^n:\, 0\leq |\af| \leq d \}, \\
\N^n_{\{d\}} &:=& \{ \af \in \N^n:\, |\af| = d \}, \quad
\N^n_{\{0,d\}} := \N^n_{\{d\}} \cup \{ 0 \}.
\earay \right.
\ee
Denote the vector of monomials:
\[
[x]_{0,m} := ( x^\af )_{ |\af| = 0, m }.
\]
The symbol $\re[x] := \re[x_1,\ldots,x_n]$
denotes the ring of polynomials in $x:=(x_1,\ldots,x_n)$
and with real coefficients,
while $\mathbb{R}[x]_d$ denotes the set of polynomials
in $\re[x]$ with degrees up to $d$.
We use $\mathbb{R}[x]_d^{hom}$  to denote the set of
homogeneous polynomials in $\re[x]$ with degree $d$.
For the complex field $\cpx$,
the $\cpx[x]$ and $\cpx[x]_d$ are similarly defined.
The $deg(p)$ denotes the total degree of a polynomial $p$.
%
%
For $t\in \re$, $\lceil t\rceil$ (resp., $\lfloor t\rfloor$)
denotes the smallest integer not smaller
(resp., the largest integer not bigger) than $t$.
%
%
%
For a matrix $A$, $A^T$ denotes its transpose.
For a symmetric matrix $X$, $X\succeq 0$ (resp., $X\succ 0$) means
$X$ is positive semidefinite (resp., positive definite).
%
%
%
The $e_i$ denotes the standard $i$th unit vector,
and $e$ is the vector of all ones.

%
%

\bigskip

In the following, we review some basics in polynomial optimization
and moment problems.
We refer to \cite{Lasserre09,Lasserre15,Laurent} for details.
A polynomial $p \in \re[x]$ is said to be a sum of squares (SOS)
if $p = p_1^2+\cdots+ p_k^2$ for some $p_1,\ldots, p_k \in \re[x]$.
The set of all SOS polynomials in $x$ is denoted as $\Sig[x]$.
For a degree $m$, denote the truncation
\[
\Sig[x]_m := \Sig[x] \cap \re[x]_m.
\]
For a tuple $g=(g_1,\ldots,g_t)$ of polynomials,
its {\it quadratic module} is the set
\[
\mbox{Qmod}(g):=  \Sig[x] + g_1 \cdot \Sig[x] + \cdots + g_t \cdot \Sig[x].
\]
The $k$th truncation of $\mbox{Qmod}(g)$ is the set
\be \label{Qk(g)}
\mbox{Qmod}(g)_k :=
\Sig[x]_{k} + g_1 \cdot \Sig[x]_{d_1} + \cdots + g_t \cdot \Sig[x]_{d_t}
\ee
where each $d_i = k - \deg(g_i)$. Note that
\[
\mbox{Qmod}(g)= \bigcup_{k\in \mathbb{N}} \mbox{Qmod}(g)_k.
\]
For a tuple $h=(h_1,\ldots,h_s)$ of polynomials,
the ideal it generates is the set
\[
\mbox{Ideal}(h) :=
h_1 \cdot \re[x] + \cdots + h_m  \cdot \re[x].
\]
The $k$th {\it truncation} of $\mbox{Ideal}(h)$ is the set
\be \label{Ik(h)}
\mbox{Ideal}(h)_{k} \, := \,
h_1 \cdot \re[x]_{k-\deg(h_1)} + \cdots + h_m  \cdot \re[x]_{k-\deg(h_m)}.
\ee
Clearly, $\mbox{Ideal}(h)=\bigcup_{k\in \mathbb{N}} \mbox{Ideal}(h)_{k}$.

Let $g,h$ be as above. Consider the set
\[
K = \{ x \in \re^n:\, h(x) = 0, \, g(x) \geq 0 \}.
\]
Clearly, if $f \in \mbox{Ideal}(h)+\mbox{Qmod}(g)$, then $f\geq 0$ on the set $K$.
The reverse is also true under certain conditions.
The set $\mbox{Ideal}(h)+\mbox{Qmod}(g)$ is said to be {\it archimedean} if
$N-\|x\|^2\in I(h)+Q(g)$ for some scalar $N>0$.

\begin{theorem} \label{thm:Put}
(\cite{Putinar})
Let $h,g,K$ be as above.
Assume $\mbox{Ideal}(h)+\mbox{Qmod}(g)$ is archimedean.
If a polynomial $f > 0$ on $K$, then
$f \in \mbox{Ideal}(h)+\mbox{Qmod}(g)$.
\end{theorem}

The above theorem is called Putinar's Positivstellensatz in the literature.
Interestingly, when $f \geq 0$ on $K$,
we also have $f \in  \mbox{Ideal}(h)+\mbox{Qmod}(g)$,
under general optimality conditions (cf.~\cite{opcd}).

\bigskip

Let $\re^{\N_{[0,d]}^n}$ be the space of multi-sequences indexed by
$\af \in \N^n_{[0,d]}$ (see the notation \reff{N[0d]}).
A vector in $\re^{\N_{[0,d]}^n}$ is called a
{\it truncated multi-sequence} (tms) of degree $d$.
Every $z \in \re^{\N_{[0,d]}^n}$ can be labelled as
\[
 z \, = \, (z_\af)_{ \af \in  \N_{[0,d]}^n }.
\]
For $t\leq d$ and $z\in \mathbb{R}^{\mathbb{N}^{n}_{d}}$,
denote the truncation:
\be \label{trun:z|0,m}
z|_{ \{0,m\} } \, := \, (z_{\alpha})_{ \af \in \N^n_{ \{0,m\} }  }.
\ee
For $ p = \Sig_{ \af \in  \N_{[0,d]}^n } p_\af x^\af \in \re[x]_d$
and $z \in \re^{\N_{[0,d]}^n}$, define the product
\be \label{df:<p,y>}
\langle p, z \rangle \, := \,
\sum_{\af \in \N_{[0,d]}^n }  p_\af z_\af.
\ee
In the above, each $p_\af$ is a coefficient.
For a polynomial $q \in \re[x]_{2k}$ and a tms
$z \in \re^{ \N^n_{[0,2k]} }$,
the product $\langle q p_1p_2, z \rangle$
is a quadratic form in the coefficients of $p_1, p_2$.
The $k$th {\it localizing matrix} of $q$,
generated by a tms $z \in \re^{\N^n_{[0,2k]}}$,
is the symmetric matrix $L_q^{(k)}(z)$ such that
\be \label{locM:q}
\langle q p_1p_2, z \rangle \, = \,
vec(p_1)^T \Big( L_q^{(k)}(z) \Big) vec(p_2)
\ee
for all $p_1,p_2 \in \re[x]$ with
$\deg(p_1), \deg(p_2) \leq  2k - \lceil \deg(q)/2 \rceil$.
In the above, $vec(p_i)$ denotes the coefficient vector of $p_i$.
When $q = 1$ (the constant one polynomial),
$L_q^{(k)}(z)$ is reduced to the so-called
{\it moment matrix} and is denoted as
\be \label{moment:mat}
M_k(z):= L_{1}^{(k)}(z).
\ee
We refer to \cite{CurtoF,Helton} for localizing and moment matrices.

\section{Odd order tensors with $\F = \re$}
\label{sc:Real:oddm}

Assume the field $\F = \re$ and the order $m$ is odd.
We discuss how to computate the real nuclear norm
$\| \mA \|_{\ast,\re}$ of a tensor $\mA \in \mt{S}^m(\re^n)$.
Note that $\lmd_i (v_i)^{\otimes m} = (-\lmd_i) (-v_i)^{\otimes m}$,
when $m$ is odd. In the decomposition of $\mA$ as in \reff{ast||A||:sym},
one can generally assume $\lmd_i \geq 0$, so
\be \label{nun:As*:odd}
\| \mA \|_{\ast,\re} = \min \left\{ \sum_{i=1}^r  \lmd_i : \,
\mA = \sum_{i=1}^r  \lmd_i (v_i)^{\otimes m}, \, \lmd_i \geq 0,
\| v_i \| = 1, v_i \in \re^n \right \}.
\ee
Let $\mathscr{B}(S)$ be the set
of Borel measures supported on the unit sphere $S$ as in \reff{usph:S}.
As pointed out in \cite{TanSha15},
$\| \mA \|_{\ast,\re}$ equals the optimal value of
\be  \label{BorOpt:R:oddm}
\left\{\baray{rl}
 \min  &  \int 1 \mt{d} \mu  \\
s.t. &  \mA = \int  x^{\otimes m} \mt{d} \mu, \,\,
  \mu \in \mathscr{B}(S).
\earay \right.
\ee
Let $\mbf{a} \in \re^{ \N^n_{ \{m\} } }$ be the vector of tensor entries of
$\mA$ such that
\be \label{a=A:af}
\mbf{a}_\af = \mA_{i_1\ldots i_m} \quad \mbox{ if } \quad
x^\af = x_{i_1} \cdots x_{i_m}.
\ee
The equality constraint in \reff{BorOpt:R:oddm} is equivalent to that
\[
\mbf{a}_\af = \int x^\af \mt{d} \mu \quad
(\af \in \N^n_{ \{m\} }).
\]
Define the cone of moments
\be \label{df:scrR:0+m}
\mathscr{R}_{ \{0,m\} } := \left\{
y \in \re^{ \N^n_{ \{0,m\} } }
\left| \baray{c}
 \exists \mu \in \mathscr{B}(S) \quad s.t. \\
\,  y_\af = \int x^\af \mt{d} \mu
\,\, \forall \,\, \af \in \N^n_{ \{0,m\} }
\earay \right.
\right\}.
\ee
The cone $\mathscr{R}_{ \{0,m\} }$ is closed, convex,
and has nonempty interior \cite[Prop.~3.2]{LMOPT}.
The optimization problem \reff{BorOpt:R:oddm} is equivalent to
\be \label{miny0:Rm(S):odd}
\left\{\baray{rl}
\min  &  (y)_0  \\
s.t. &  (y)_\af = \mbf{a}_\af \,\, \big( \af \in \N^n_{ \{m\} } \big), \\
 &  y \in \mathscr{R}_{ \{0,m\} }.
\earay \right.
\ee

\subsection{An algorithm}

The cone $\mathscr{R}_{ \{0,m\} }$ can be approximated
by semidefinite relaxations. Denote the cones
\begin{align}
\label{scr(S):2k}
\mathscr{S}^{2k} & := \left\{
z \in  \left. \re^{ \N^n_{ [0,2k] } }  \right|
M_k(z) \succeq 0, \, L^{(k)}_{1-\|x\|^2}(z) = 0
\right\}, \\
\label{scr(S):2k:0+m}
\mathscr{S}^{2k}_{ \{0,m\} } & := \left\{
y \in \left. \re^{ \N^n_{ \{0,m\} } }  \right|
\exists \, z \in  \mathscr{S}^{2k}, \,  \, y = z|_{ \{0,m \} }
\right\}.
\end{align}
It can be shown that (cf.~\cite[Prop.~3.3]{LMOPT})
\be \label{SDr:R0,m:odd}
\mathscr{R}_{ \{0,m\} } = \bigcap_{k \geq m/2 }
\mathscr{S}^{2k}_{ \{0,m\} }.
\ee
This leads to the hierarchy of semidefinite relaxations
\be \label{min(y)0:mom:k}
\left\{ \baray{rl}
\| \mA \|_{k\ast,\re}  \, := \, \min\limits_{ z } &  (z)_0 \\
s.t. &  (z)_\af = \mbf{a}_\af \,\,(\af \in \N^n_{ \{m\} } ), \\
 &  z \in \mathscr{S}^{2k},
\earay \right.
\ee
for the relaxation orders
$k = m_0, m_0 +1, \ldots$, where $m_0:=\lceil m/2 \rceil$.
Since $\mathscr{R}_{ \{0,m\} } \subseteq \mathscr{S}^{2k+2}
\subseteq \mathscr{S}^{2k}$ for all $k$,
we have the monotonicity relationship
\be  \label{mrel:rhok:om}
\| \mA \|_{m_0\ast,\re}
 \leq \cdots \leq  \| \mA \|_{k\ast,\re}   \leq \cdots \leq
\| \mA \|_{\ast,\re}.
\ee

\bigskip

\begin{alg} \label{alg:R:odd}
Given a tensor $\mA \in \mt{S}^m(\re^n)$ with odd $m$,
let $k = m_0$ and do:
\bit

\item [Step 1] Solve the semidefinite relaxation \reff{min(y)0:mom:k},
for an optimizer $z^k$.

\item [Step 2]  Let $y^k := z^k|_{ \{0,m\} }$
(see \reff{trun:z|0,m} for the truncation).
Check whether $y^k \in \mathscr{R}_{ \{0,m\} }$ or not.
If yes, then $\| \mA \|_{\ast, \re} = \| \mA \|_{k\ast,\re} $ and go to Step~3;
otherwise, let $k :=k+1$ and go to Step~1.

\item [Step 3] Compute the decomposition of $y^k$ as
\[
y^k =  \lmd_1 [v_1]_{0,m} + \cdots + \lmd_r [v_r]_{0,m}
\]
with all $\lmd_i >0, v_i \in S$. This gives the nuclear decomposition
\[
\mA =  \lmd_1 (v_1)^{\otimes m} + \cdots + \lmd_r (v_r)^{\otimes m}
\]
such that $\lmd_1 + \cdots + \lmd_r = \| \mA \|_{\ast,\re}$.

\eit

\end{alg}

In the above, the method in \cite{ATKMP} can be applied to
check whether $y^k \in \mathscr{R}_{ \{0,m\} }$ or not.
If yes, a nuclear decomposition can also be obtained.
It requires to solve a moment optimization problem
whose objective is randomly generated.

\subsection{Convergence properties}

The dual cone of the set $\mathscr{R}_{ \{0,m\} }$ is
\be
\mathscr{P}(S)_{0,m} := \{
t + q \mid \, t \in \re, \, q \in \re[x]_m^{hom}, \,
t + q \geq 0 \, \mbox{ on } S
\}.
\ee
So, the dual optimization problem of \reff{miny0:Rm(S):odd} is
\be \label{max<p,A>:1-p>=0}
\max \limits_{p \in \re[x]_m^{hom} }  \quad
\langle p, \mbf{a} \rangle  \quad
s.t. \quad  1 - p  \in \mathscr{P}(S)_{0,m}.
\ee

\begin{lemma} \label{nng:opval:odm}
Let $\mbf{a}$ be the vector as in \reff{a=A:af}.
Then, both \reff{miny0:Rm(S):odd} and \reff{max<p,A>:1-p>=0}
achieve the same optimal value, which equals the nuclear norm
$\| \mA \|_{\ast,\re}$.
\end{lemma}
\begin{proof}
Clearly, $p=0$ (the zero polynomial) is an interior point
of \reff{max<p,A>:1-p>=0}.
By the linear conic duality theory~\cite[\S2.4]{BTN},
\reff{miny0:Rm(S):odd} and \reff{max<p,A>:1-p>=0}
have the same optimal value which is $\| \mA \|_{\ast,\re}$,
and \reff{miny0:Rm(S):odd} achieves it.
The feasible set of \reff{max<p,A>:1-p>=0} is compact.
This is because $|p| \leq 1$ on the unit sphere
and $p$ is a form of degree $m$.
So, \reff{max<p,A>:1-p>=0} also achieves its optimal value,
which equals $\| \mA \|_{\ast,\re}$.
\end{proof}

Denote the nonnegative polynomial cones:
\be \label{Qk:oddm}
Q_k \, := \, \mbox{Ideal}_{2k}(1-\|x\|^2) + \Sig[x]_{2k}, \quad
Q := \bigcup_{k\geq 1} Q_k.
\ee
The cones $Q_k$ and $\mathscr{S}^{2k}$
are dual to each other (cf.~\cite{LMOPT}).
So, the dual optimization problem of \reff{min(y)0:mom:k} is
\be \label{m<pf>:Qk:oddm}
\max \limits_{p \in \re[x]_m^{hom} } \quad
 \langle p, \mbf{a} \rangle  \quad s.t. \quad 1 - p  \in Q_k.
\ee
Some properties of Lasserre relaxations
were mentioned in \cite{TanSha15}. For completeness of the paper,
we give the properties with more details and rigorous proofs.

\begin{lemma} \label{ach:opv:k:om}
Let $\mbf{a}$ be the vector as in \reff{a=A:af}.
Then, both \reff{min(y)0:mom:k} and \reff{m<pf>:Qk:oddm}
achieve the same optimal value, which equals
$\| \mA \|_{k\ast,\re}$.
Moreover, for each $k\geq m_0$,
$\| \mA \|_{k\ast,\re}$ is a norm function in
$\mA \in \mt{S}^m(\re^n)$.
\end{lemma}
\begin{proof}
For each $k \geq m_0$, $p=0$ is an interior point of \reff{m<pf>:Qk:oddm}.
So, \reff{min(y)0:mom:k} and \reff{m<pf>:Qk:oddm}
have the same optimal value $\| \mA \|_{k\ast,\re} $,
and \reff{min(y)0:mom:k} achieves it (cf.~\cite[\S2.4]{BTN}).
The set $Q_k$ is closed, which can be implied by
Theorem~3.1 of \cite{Marsh03}(
aslo see Theorem~3.35 of \cite{Laurent}).
When $1-p \in Q_k$, $|p| \leq 1$ on the unit sphere $S$.
Hence, the feasible set of \reff{m<pf>:Qk:oddm} is compact,
and it also achieves its optimal value.
In the following, we prove that
$\| \mA \|_{k\ast,\re}$ is a norm function in $\mA$.
\bit
\item [1)] Because $M_k(z) \succeq 0$, $(z)_0 \geq 0$.
So $\| \mA \|_{k\ast,\re} \geq 0$ for all $\mA$.

\item [2)] Let $z^*$ be an optimizer such that
$\| \mA \|_{k\ast,\re} = (z^*)_0$.
If $\| \mA \|_{k\ast,\re} = 0$, then $(z^*)_0=0$
and $z^*=0$, because $M_k(z^*)\succeq 0$ and
$L_{1-\|x\|^2}^{(k)}(z^*)=0$.
So, $\mbf{a}=0$ and $\mA$ must be the zero tensor.

\item [3)] First, we show that
$\| -\mA \|_{k\ast,\re} = \| \mA \|_{k\ast,\re}$.
For $z \in \re^{ \N^n_{[0,2k]}}$,
define $s(z) \in \re^{ \N^n_{[0,2k]}}$ be such that
\[
 ( s(z) )_\af =  (-1)^{|\af|} (z)_\af, \quad
 \forall \, \af \in \N^n_{[0,2k]}.
\]
One can verify that ($\mbf{1}$ denotes the vector of all ones)
\[
M_k( s(z) ) =  \mbox{diag}( [-\mbf{1}]_k ) M_k( z )
\mbox{diag}( [-\mbf{1}]_k ),
\]
\[
L_{1-\|x\|^2}^{(k)}( s(z) ) =
\mbox{diag}( [-\mbf{1}]_{k-1} ) L_{1-\|x\|^2}^{(k)}( z )
\mbox{diag}( [-\mbf{1}]_{k-1}  ).
\]
Thus, $z$ is feasible for \reff{min(y)0:mom:k} with tensor $\mA$
if and only if $s(z)$ is feasible
for \reff{min(y)0:mom:k} with tensor $-\mA$.
Since $(z)_0 = (s(z))_0$, we get
$\| -\mA \|_{k\ast,\re} = \| \mA \|_{k\ast,\re}$.

Second, we show that
$\| t\mA \|_{k\ast,\re} = t \| \mA \|_{k\ast,\re}$ for all $t>0$.
For $t>0$, $z$ is feasible
for \reff{min(y)0:mom:k} with tensor $\mA$
if and only if $tz$ is feasible
for \reff{min(y)0:mom:k} with tensor $t\mA$.
Note that $t (z)_0 = (tz)_0$ for $t>0$.
So, $\| t\mA \|_{k\ast,\re} = t \| \mA \|_{k\ast,\re}$ for $t>0$.

The above two cases imply that
\[
\| t \mA \|_{k\ast,\re} =|t| \cdot \| \mA \|_{k\ast,\re} \, \quad
\forall \, \mA \in \mt{S}^m( \re^n), \, \, \forall \, t \in \re.
\]

\item [4)] The feasible set of \reff{min(y)0:mom:k}
is a convex set in $(z,\mA)$. Its objective is a linear function in $z$.
By the result in \cite[\S3.2.5]{BVbook},
$\| \mA \|_{k\ast,\re}$ is a convex function in $\mA$, so
$
\| \mA + \mB \|_{k\ast,\re} \leq \| \mA \|_{k\ast,\re} +
\| \mB \|_{k\ast,\re}
$
for all $\mA,\mB$.

\eit
\end{proof}

The convergence of Algorithm~\ref{alg:R:odd} is summarized as follows.

\begin{theorem} \label{thm:cvg:oddm}
Let $\| \mA \|_{k\ast,\re}$ be the optimal value of \reff{min(y)0:mom:k}.
For all $\mA \in \mt{S}^m(\re^n)$,
Algorithm~\ref{alg:R:odd} has the following properties:
\bit
\item [(i)]
$\lim\limits_{k\to\infty} \| \mA \|_{k\ast,\re}  = \| \mA \|_{\ast,\re}$.

\item [(ii)] Let $p^*$ be an optimizer of \reff{max<p,A>:1-p>=0}.
If $1-p^* \in Q$, then
$\| \mA \|_{k\ast,\re} = \| \mA \|_{\ast,\re}$ for all $k$ sufficiently big.

\item [(iii)] If $y^k \in \mathscr{R}_{ \{0,m\} }$ for some order $k$,
then $\| \mA \|_{\ast,\re}  = \| \mA \|_{k\ast,\re}$.

\item [(iv)] The sequence $\{ y^k \}_{k=m_0}^{\infty}$
converges to a point in $\mathscr{R}_{ \{0,m\} }$.

\eit
\end{theorem}
\begin{proof}
(i) By Lemma~\ref{nng:opval:odm},
for every $\eps>0$, there exists $p_1 \in \re[x]_m^{hom}$ such that
\[
1 - p_1 >0  \mbox{ on }  S, \qquad
\langle p_1, \mbf{a} \rangle \geq \| \A \|_{\ast,\re} - \eps.
\]
By Theorem~\ref{thm:Put}, there exists $k_1$ such that
$
 1 - p_1  \in Q_{k_1}.
$
By Lemma~\ref{ach:opv:k:om}, we get
\[
\| \mA \|_{k_1\ast,\re}    \geq  \| \A \|_{\ast,\re} - \eps.
\]
The monotonicity relation \reff{mrel:rhok:om} and the above imply that
\[
\| \mA \|_{\ast,\re} \geq \lim\limits_{k\to\infty} \| \mA \|_{k\ast,\re}
\geq \| \mA \|_{\ast,\re} - \eps.
\]
Since $\eps>0$ can be arbitrarily small,
the item (i) follows directly.

(ii) If $1-p^* \in Q$, then
$1-p^* \in Q_{k_2}$ for some $k_2$.
By Lemma~\ref{nng:opval:odm}, we know
\[
\| \mA \|_{\ast,\re} = \langle p^*, \mbf{a} \rangle
\leq \| \mA \|_{k_2\ast,\re} .
\]
Then, \reff{mrel:rhok:om} implies that
$\| \mA \|_{k\ast,\re} = \| \mA \|_{\ast,\re}$
for all $k \geq k_2$.

(iii) If $y^k \in \mathscr{R}_{ \{0,m\} }$ for some $k$,
then $\| \mA \|_{k\ast,\re}  \geq \| \mA \|_{\ast,\re}$,
by Lemmas~\ref{nng:opval:odm} and \ref{ach:opv:k:om}.
Then, the equality $\| \mA \|_{k\ast,\re} = \| \mA \|_{\ast,\re}$
follows from \reff{mrel:rhok:om}.

(iv) Note the relations
\[
(y^k)_0 = \| \mA \|_{k\ast,\re}, \quad
(y^k)_\af  = \mbf{a}_{ \af } \quad
(\forall \, \af \in \N^n_{ \{ m \} } ).
\]
Since $\| \mA \|_{k\ast,\re}  \to  \| \mA \|_{\ast,\re}$,
we know the limit $y^*$ of the sequence $\{ y^k \}$ must exist.
For all $k\geq m/2$, we have
$y^k \in \mathscr{S}^{2k}_{ \{0,m\} }$.
The distance between $\mathscr{S}^{2k}_{ \{0,m\} }$ and
$\mathscr{R}_{ \{0,m\} }$ tends to zero as $k\to \infty$
(cf.~\cite[Prop.~3.4]{LMOPT}),
so $y^* \in  \mathscr{R}_{ \{0,m\} }$.
It can also be implied by the equality \reff{SDr:R0,m:odd}.
\end{proof}

In Theorem~\ref{thm:cvg:oddm}(ii), we always have $1-p^*\geq 0$ on $S$.
Under some general conditions,
we further have $1-p^* \in Q$, as shown in \cite{opcd}.
Thus, Algorithm~\ref{alg:R:odd} usually has finite convergence,
which is confirmed by numerical experiments in \S\ref{sc:num}.

\section{Even order tensors with $\F = \re$}
\label{sc:Real:meven}

Assume the order $m$ is even and the field $\F = \re$.
For a symmetric tensor $\mA \in \mt{S}^m(\re^n)$,
the sign of $\lmd_i$ in \reff{ast||A||:sym}
cannot be generally assumed to be positive.
However, we can always decompose $\mA$ as
($\mbf{1}$ is the vector of all ones)
\be \label{dcF:v+v-:lmd>=0}
\left\{ \baray{c}
\mA = \sum_{i=1}^{r_1}  \lmd_i^+ (v_i^+)^{\otimes m} -
\sum_{i=1}^{r_2}  \lmd_i^- (v_i^-)^{\otimes m}, \\
\lmd_i^+ \geq 0, \, \| v_i^+ \|  = 1, \,
\mathbf{1}^Tv_i^+ \geq 0, \, v_i^+ \in \re^n,\\
\lmd_i^- \geq 0, \, \| v_i^- \| = 1, \,
\mathbf{1}^Tv_i^- \geq 0, \, v_i^- \in \re^n.
\earay \right.
\ee
Let $\mathscr{B}(S^+)$ be the set of Borel measures
supported in the half unit sphere
\be \label{set:S+}
S^+ := \{ x\in \re^n \, \mid \, \|x\|=1, \mathbf{1}^Tx \geq 0 \}.
\ee
Clearly, the weighted Dirac measures
\[
\mu^+ := \sum_{i=1}^{r_1}  \lmd_i^+ \dt_{v_i^+}, \quad
\mu^- := \sum_{i=1}^{r_2}  \lmd_i^- \dt_{v_i^-}
\]
belong to $\mathscr{B}(S^+)$.
%
%
The decomposition \reff{dcF:v+v-:lmd>=0} is equivalent to
\be \label{dcpA:mu+mu-}
\mA = \int  x^{\otimes m} \mt{d} \mu^+
 - \int  x^{\otimes m} \mt{d} \mu^-.
\ee
Reversely, if there exist $\mu^+, \mu^- \in \mathscr{B}(S^+)$
satisfying \reff{dcpA:mu+mu-}, then $\mA$ has a decomposition
as in \reff{dcF:v+v-:lmd>=0} (cf.~\cite[Prop.~3.3]{ATKMP}).
Therefore, the nuclear norm
$\| \mA \|_{\ast,\re}$ equals the optimal value of the problem
\be \label{nnF:opt:mu+-}
\left\{\baray{rl}
\min  & \int 1 \mt{d} \mu^+ +  \int 1 \mt{d} \mu^-\\
s.t. &  \mA = \int  x^{\otimes m} \mt{d} \mu^+
- \int  x^{\otimes m} \mt{d} \mu^-, \\
&  \mu^+, \mu^- \in \mathscr{B}(S^+).
\earay\right.
\ee
Let $\mbf{a} \in \re^{ \N^n_{ \{m\} } }$ be the vector such that
\be \label{mbf(a):evm}
 \mbf{a}_\af  \, = \, \mA_{i_1\cdots i_m}
 \quad \mbox{ if } \quad  x^\af = x_{i_1}\cdots x_{i_m}.
\ee
Denote the cone of moments
\be \label{scr(R)+:0+m}
\mathscr{R}^+_{ \{0,m\} } := \left\{
y \in \re^{ \N^n_{ \{0, m\} } }
\left| \baray{c}
 \exists \mu \in \mathscr{B}(S^+) \, \mbox{ such that } \\
 y_\af = \int x^\af \mt{d} \mu \, \mbox{ for } \, \af \in \N^n_{ \{0, m\} }
\earay \right.
\right\}.
\ee
Then, \reff{nnF:opt:mu+-} is equivalent to
\be \label{miny0:Rm(S)}
\left\{ \baray{rl}
\min   &  (y^+)_0 + (y^-)_0 \\
s.t. &   (y^+)_\af - (y^-)_\af  = \mbf{a}_\af \,\, ( \af \in \N^n_{\{m\} } ), \\
 & y^+, y^- \in \mathscr{R}^+_{ \{0,m\} }.
\earay \right.
\ee

\subsection{An algorithm}

The cone $\mathscr{R}^+_{ \{0,m\} }$ can be approximated
by semidefinite relaxations.
%
%
Denote the cones
\begin{align}
\mathscr{S}^{+,2k} & := \left\{
z \in  \left. \re^{ \N^n_{ [0, 2k] } }  \right|
M_k(z) \succeq 0, \, L^{(k)}_{\mbf{1}^Tx}(z) \succeq 0,
\, L^{(k)}_{1-\|x\|^2}(z)=0
\right\}, \\
\mathscr{S}^{+,2k}_{ \{0,m\} } & := \left\{
y \in  \left.  \re^{ \N^n_{ \{0, m\} } }  \right|
\exists \, z \in  \mathscr{S}^{+,2k}, \,  \, y = z|_{ \{0,m \} }
\right\}.
\end{align}
Note that $\mathscr{S}^{+,2k}_{ \{0,m\} }$ is a projection of
$\mathscr{S}^{+,2k}$ and
$
\mathscr{R}^{+}_{ \{0,m\} } \subseteq \mathscr{S}^{+,2k}_{ \{0,m\} }
$
for all $k$. As shown in \cite{LMOPT}, it holds that
\be \label{SDr:R0,m:S+}
\mathscr{R}^+_{ \{0,m\} } = \bigcap_{k \geq m/2 }
\mathscr{S}^{+,2k}_{ \{0,m\} }.
\ee
So, we get the hierarchy of semidefinite relaxations
for solving \reff{miny0:Rm(S)}:
\be \label{rho(k):m:even}
\left\{ \baray{rl}
\| \mA \|_{k\ast,\re}  \, := \, \min\limits_{z^+,  z^-} &  (z^+)_0 + (z^-)_0 \\
s.t. &   (z^+)_\af - (z^-)_\af  = \mbf{a}_\af \, ( \af \in \N^n_{\{m\} } ), \\
 &  z^+, z^- \in \mathscr{S}^{+,2k},
\earay \right.
\ee
for $k= m_0, m_0+1, \ldots$ ($m_0 = \lceil m/2 \rceil$).
Similar to \reff{mrel:rhok:om},
we also have the monotonicity relationship
\be  \label{rhok:mcr:evm}
\| \mA \|_{m_0\ast,\re} \leq \cdots
\leq \| \mA \|_{k\ast,\re} \leq \cdots \leq \| \mA \|_{\ast,\re}.
\ee

\begin{alg} \label{alg:even:m}
For a given tensor $\mA \in \mt{S}^m(\re^n)$,
let $k = m_0$ and do:
\bit

\item [Step 1] Solve the semidefinite relaxation \reff{rho(k):m:even},
for an optimizer $(z^{+,k}, z^{-,k})$.

\item [Step 2]  Let $y^{+,k} := z^{+,k}\big|_{ \{0,m\} }$,
$y^{-,k} := z^{-,k}\big|_{ \{0,m\} }$
(see \reff{trun:z|0,m} for the truncation).
Check whether $y^{+,k}, y^{-,k} \in \mathscr{R}^+_{ \{0,m\} }$ or not.
If they both belong, then
$\| \mA \|_{\ast, \re} = \| \mA \|_{k\ast,\re}$ and go to Step~3;
otherwise, let $k :=k+1$ and go to Step~1.

\item [Step 3] Compute the decompositions of
$y^{+,k}, y^{-,k}$ as
\[
y^{+,k} = \sum_{i=1}^{r_1} \lmd_i^+ [v_i^+]_{0,m}, \quad
y^{-,k} = \sum_{i=1}^{r_2} \lmd_i^- [v_i^-]_{0,m},
\]
with all $\lmd^+_i >0, \lmd_i^->0$ and $v_i^+, v_i^- \in S^+$.
The above gives the nuclear decomposition:
\[
\mA = \sum_{i=1}^{r_1} \lmd_i^+ (v_i^+)^{\otimes m} -
\sum_{i=1}^{r_2} \lmd_i^-  (v_i^-)^{\otimes m}
\]
such that
$\sum_{i=1}^{r_1} \lmd_i^+ + \sum_{i=1}^{r_2} \lmd_i^- = \| \mA \|_{\ast,\re}$.

\eit

\end{alg}

In the above, the method in \cite{ATKMP}
can be applied to check if
$y^{+,k}, y^{-,k} \in \mathscr{R}^+_{ \{0,m\} }$ or not.
If yes, a nuclear decomposition can also be obtained.
In Step~3, it is possible that $r_1=0$ or $r_2 = 0$,
for which case the corresponding $y^{+,k}$ or $y^{-,k}$
is the vector of all zeros.
Note that Algorithm~\ref{alg:even:m} can also be applied
to compute $\| \mA \|_{\ast,\re}$
even if the order $m$ is odd.

\subsection{Convergence properties}

The dual cone of the set $\mathscr{R}^+_{ \{0,m\} }$ is
\[
\mathscr{P}(S^+)_{0,m}  \,:= \, \{ t + p \mid
t \in \re, \, p \in \re[x]_m^{hom}, \, t + p \geq 0  \mbox{ on } S^+  \}.
\]
So, the dual optimization problem of \reff{miny0:Rm(S)} is
\be \label{max<p,f>:1>=|p|}
 \max \limits_{ p \in  \re[x]_m^{hom} } \quad  \langle p, \mbf{a} \rangle
\quad s.t. \quad 1 \pm p \in \mathscr{P}(S^+)_{0,m}.
\ee

\begin{lemma} \label{mnng:val=:evm}
Let $\mbf{a}$ be the vector as in \reff{mbf(a):evm}.
Then, both \reff{miny0:Rm(S)} and \reff{max<p,f>:1>=|p|}
achieve the same optimal value which equals $\| \mA \|_{\ast, \re}$.
\end{lemma}
\begin{proof}
The feasible set of \reff{miny0:Rm(S)} is always nonempty,
say, $(\hat{y}^+, \hat{y}^-)$ is a feasible pair.
Let $\xi$ be an interior point of $\mathscr{R}^+_{ \{0,m\} }$.
Then $\hat{y}^+ +\xi, \hat{y}^- + \xi$ are both interior points of
$\mathscr{R}^+_{ \{0,m\} }$.
The zero polynomial $p=0$ is an interior point of \reff{max<p,f>:1>=|p|}.
By the linear conic duality theory \cite[\S2.4]{BTN},
the optimal values of \reff{miny0:Rm(S)} and \reff{max<p,f>:1>=|p|}
are equal, and they both achieve it.
The optimal value of \reff{miny0:Rm(S)}
is $\| \mA \|_{\ast, \re}$,
so it is also the optimal value of \reff{max<p,f>:1>=|p|}.
\end{proof}

Next, we study the properties of the relaxation~\reff{rhok:mcr:evm}.
Denote the cones of nonnegative polynomials:
\be \label{Qk+:evm}
Q_k^+ \, := \, \mbox{Ideal}_{2k}(1-\| x \|^2)
+ \mbox{Qmod}_{2k}(\mathbf{1}^Tx), \quad
Q^+ \, := \, \bigcup_{ k \geq 1 } Q_k^+.
\ee
The cones $Q_k^+$ and $\mathscr{S}^{+,2k}$
are dual to each other (cf.~\cite{LMOPT}),
so the dual optimization problem of \reff{rho(k):m:even} is
\be \label{mx<p,f>:1+-p:Qk}
 \max \limits_{ p \in  \re[x]_m^{hom} }
\quad \langle p, \mbf{a} \rangle \quad
s.t. \quad 1 \pm p  \in Q_k^+.
\ee

\begin{lemma} \label{achval:Qk:evm}
Let $\mbf{a}$ be the vector of entries of $\mA$ as in \reff{mbf(a):evm}.
For each $k \geq m_0$, both \reff{rho(k):m:even} and \reff{mx<p,f>:1+-p:Qk}
achieve the same optimal value $\| \mA \|_{k\ast,\re}$.
Moreover, $\| \mA \|_{k\ast,\re}$ is a norm function in
$\mA \in \mt{S}^m(\re^n)$.
\end{lemma}
\begin{proof}
The zero form $p=0$ is an interior point
of \reff{mx<p,f>:1+-p:Qk}, for all $k \geq m_0$.
By the linear conic duality theory \cite[\S2.4]{BTN},
\reff{rho(k):m:even} and \reff{mx<p,f>:1+-p:Qk}
have the same optimal value and
\reff{rho(k):m:even} achieves it.
The vanishing ideal of $S^+$ is $\mbox{Ideal}(1-\|x\|^2)$,
so the set $Q_k^+$ is closed
(cf.~\cite[Theorem~3.35]{Laurent} or \cite[Theorem~3.1]{Marsh03}).
When $p$ is feasible for \reff{mx<p,f>:1+-p:Qk},
$|p| \leq 1$ on the unit sphere $S$.
So, the feasible set of \reff{mx<p,f>:1+-p:Qk}
is compact, and it also achieves its optimal value.
As in the proof of Lemma~\ref{ach:opv:k:om},
we can similarly prove that
$\| \mA \|_{k\ast,\re}$ is a norm function in $\mA$, as follows:
\bit
\item [1)] Because $(z^+)_0 \geq 0$, $(z^-)_0 \geq 0$,
we must have $\| \mA \|_{k\ast,\re} \geq 0$ for all $\mA$.

\item [2)] Let $(z^{+*}, z^{-*})$ be such that
$\| \mA \|_{k\ast,\re} = (z^{+*})_0 + (z^{-*})_0$.
If $\| \mA \|_{k\ast,\re} = 0$, then
$(z^{+*})_0 = (z^{-*})_0=0$, and hence
and $z^{+*}=z^{-*}=0$.
So, $\mA$ must be the zero tensor.

\item [3)] Let $s(z)$ be the function
as in the proof of Lemma~\ref{ach:opv:k:om}.
One can similarly prove that
$(z^+,z^-)$ is feasible for \reff{rho(k):m:even} with tensor $\mA$
if and only if $(s(z^+), s(z^-))$ is feasible
\reff{rho(k):m:even} with tensor $-\mA$.
This implies that $\| -\mA \|_{k\ast,\re} = \| \mA \|_{k\ast,\re}$.
Similarly, one can show that
$\| t\mA \|_{k\ast,\re} = t \| \mA \|_{k\ast,\re}$ for $t>0$.
Therefore,
$\| t \mA \|_{k\ast,\re} =|t| \cdot \| \mA \|_{k\ast,\re}$
for all $\mA$ and for all $t \in \re$.

\item [4)] For all tensors $\mA, \mB$, the triangular inequality
$
\| \mA + \mB \|_{k\ast,\re} \leq \| \mA \|_{k\ast,\re} +
\| \mB \|_{k\ast,\re}
$
follows from the fact that
the feasible set of \reff{rho(k):m:even}
is a convex set in $(z,\mA)$ and its objective is linear in $z$.

\eit

\end{proof}

The convergence properties of Algorithm~\ref{alg:even:m} are as follows.

\begin{theorem} \label{thm:cvg:evm}
Let $\| \mA \|_{k\ast,\re}$ be the optimal value of \reff{rho(k):m:even}.
For all $\mA \in \mt{S}^m(\re^n)$,
Algorithm~\ref{alg:even:m} has the following properties:
\bit
\item [(i)]
$\lim\limits_{k\to\infty} \| \mA \|_{k\ast,\re} = \| \mA \|_{\ast,\re}$.

\item [(ii)] Let $p^*$ be an optimizer of \reff{max<p,f>:1>=|p|}.
If $1 \pm p^* \in Q^+$, then
$\| \mA \|_{k\ast,\re} = \| \mA \|_{\ast,\re}$ for all $k$ sufficiently big.

\item [(iii)] If $y^{+,k}, y^{-,k} \in \mathscr{R}^+_{ \{0,m\} }$ for some order $k$,
then $\| \mA \|_{k\ast,\re} = \| \mA \|_{\ast,\re}$.

\item [(iv)] The sequence
$\{ (y^{+,k},  y^{-,k}) \}_{k=m_0}^\infty$
is bounded, and for its each accumulation point
$(\hat{y}^+, \hat{y}^-)$, we must have
\[
\hat{y}^+, \hat{y}^- \in \mathscr{R}^{+}_{ \{0,m\} }, \quad
( \hat{y}^+)_0 + ( \hat{y}^- )_0 = \| \mA \|_{\ast,\re}.
\]
Moreover, if the nuclear decomposition of $\mA$ over $\re$ is unique,
then $(y^{+,k},  y^{-,k})$ converges to the pair
$(\hat{y}^+, \hat{y}^-)$ as above.
\eit

\end{theorem}
\begin{proof}
(i)
By Lemma~\ref{mnng:val=:evm},
for every $\eps>0$, there exists $p_1 \in \re[x]_m^{hom}$ such that
\[
1 \pm p_1 >0  \mbox{ on }  S, \qquad
\langle p_1, \mbf{a} \rangle \geq \| \A \|_{\ast,\re} - \eps.
\]
By Theorem~\ref{thm:Put}, there exists $k_1$ such that
$
 1 \pm p_1  \in Q_{k_1}^+ .
$
By Lemma~\ref{achval:Qk:evm}, we can get
\[
\| \mA \|_{k_1\ast,\re}   \geq  \| \A \|_{\ast,\re} - \eps.
\]
The relation \reff{rhok:mcr:evm} and the above imply that
\[
\| \mA \|_{\ast,\re} \geq \lim\limits_{k\to\infty} \| \mA \|_{k\ast,\re}
\geq \| \mA \|_{\ast,\re} - \eps.
\]
The item (i) follows from that
$\eps>0$ can be arbitrarily small.

(ii)-(iii): The proof is the same
as for Theorem~\ref{thm:cvg:oddm} (ii)-(iii),
by using Lemmas~\ref{mnng:val=:evm} and \ref{achval:Qk:evm}.

(iv) Note that $y^{+,k} = z^{+,k}\big|_{ \{0,m\} }$,
$y^{-,k} = z^{-,k}\big|_{ \{0,m\} }$,
\[
M_k( z^{+,k} ) \succeq 0, \quad M_k( z^{-,k} ) \succeq 0,
\]
and $(y^{+,k})_0 + (y^{-,k})_0 = \| \mA \|_{k\ast,\re} \leq
\| \mA \|_{\ast,\re}$ for all $k$. From the condition
\[
L^{(k)}_{1-\|x\|^2} ( z^{+,k} )  =
L^{(k)}_{1-\|x\|^2} ( z^{-,k} )  = 0,
\]
one can see that the sequence of diagonal entries of
$M_k( z^{+,k} ),\, M_k( z^{-,k} )$
is bounded. Then, we can show that the sequence
$\{ (z^{+,k}, z^{-,k} ) \}$ is bounded. This implies that
$\{ (y^{+,k},  y^{-,k}) \}_{k=m_0}^\infty$
is also bounded. When $(\hat{y}^+, \hat{y}^-)$ is one of
its accumulation points, we can get
$
( \hat{y}^+)_0 + ( \hat{y}^- )_0 = \| \mA \|_{\ast,\re}
$
by evaluating the limit. Note that
$y^{+,k}, y^{-,k} \in \mathscr{S}^{+,2k}_{ \{0,m\} }$
for all $k$. The distance between $\mathscr{S}^{+,2k}_{ \{0,m\} }$ and
$\mathscr{R}^+_{ \{0,m\} }$ tends to zero as $k\to \infty$
(cf.~\cite[Prop.~3.4]{LMOPT}),
so we have $\hat{y}^+, \hat{y}^- \in  \mathscr{R}^+_{ \{0,m\} }$.
It can also be implied by \reff{SDr:R0,m:S+}.
Next, write down the decompositions:
\[
\hat{y}^+ = \sum_{i=1}^{r_1} \lmd_i^+ [v_i^+]_{0,m}, \quad
\hat{y}^- = \sum_{i=1}^{r_2} \lmd_i^- [v_i^-]_{0,m},
\]
with all $\lmd_i^+ \geq 0$, $\lmd_i^- \geq 0$,
and $v_i^+, v_i^- \in S^+$. They give the real nuclear
decomposition $\mA = \mA_1 - \mA_2$, with
\[
\mA_1 = \sum_{i=1}^{r_1} \lmd_i^+ (v_i^+)^{\otimes m}, \quad
\mA_2 = \sum_{i=1}^{r_2} \lmd_i^- (v_i^-)^{\otimes m}.
\]
When the nuclear decomposition of $\mA$ is unique,
the decompositions of $\mA_1, \mA_2$ are also unique.
So, the accumulation point $(\hat{y}^+, \hat{y}^-)$ is unique
and $(y^{+,k},  y^{-,k})$ must converge to it as $k \to \infty$.
\end{proof}

In Theorem~\ref{thm:cvg:evm}(ii), we always have
$1 \pm p^* \geq 0$ on $S^+$.
Under some general optimality conditions,
it holds that $1 \pm p^* \in Q^+$.
So, Algorithm~\ref{alg:even:m} generally has finite convergence.
This is confirmed by numerical experiments in \S\ref{sc:num}.

\section{Nuclear norms with $\F = \cpx$}
\label{sc:cpx}

When the ground field $\F = \cpx$, the nuclear norm of
$\mA \in \mt{S}^m(\cpx^n)$ is
\be \label{A:s*:C}
\| \mA \|_{\ast,\cpx} = \min \left\{ \sum_{i=1}^r |\lmd_i|  : \,
\mA = \sum_{i=1}^r  \lmd_i (w_i)^{\otimes m},
\| w_i \| = 1, w_i \in \cpx^n \right \}.
\ee
First, we formulate an optimization problem
for computing $\| \mA \|_{\ast,\cpx}$.

\begin{lemma} \label{lm:opt:cnuc}
For all $\mA \in \mt{S}^m(\cpx^n)$, $\| \mA \|_{\ast,\cpx}$
equals the optimal value of
\be \label{cnn:F=ui+vi}
\left\{ \baray{rl}
  \min & \sum_{i=1}^r  \lmd_i   \\
s.t. & \mA = \sum_{i=1}^r  \lmd_i (u_i+\sqrt{-1}v_i)^{\otimes m}, \\
 &  \lmd_i \geq 0, \, \| u_i \|^2 + \| v_i \|^2 =1, \, u_i, v_i \in \re^n,  \\
 & \mbf{1}^T v_i \geq 0, \,
\sin(\frac{2\pi}{m})\mbf{1}^T u_i - \cos(\frac{2\pi}{m}) \mbf{1}^T v_i  \geq 0.
\earay \right.
\ee
In the above, $\mbf{1}$ is the vector of all ones.
\end{lemma}
\begin{proof}
The decomposition of $\mA$ as in \reff{A:s*:C} is equivalent to
\[
\mA = \sum_{i=1}^r  \lmd_i (\tau_i w_i)^{\otimes m},
\]
for all unitary $\tau_i \in \cpx$ with $\tau_i^m = 1$. Write
\[
w_i = u_i + \sqrt{-1} v_i, \quad u_i, v_i \in \re^n,
\]
then
\[
\baray{rcl}
\mbf{1}^T w_i &=& (\mbf{1}^T u_i) + \sqrt{-1} (\mbf{1}^T v_i), \\
\mbf{1}^T (\tau_i w_i ) &=& \tau_i
\Big( (\mbf{1}^T u_i) + \sqrt{-1} (\mbf{1}^T v_i) \Big).
\earay
\]
Write $\mbf{1}^T w_i = r e^{\sqrt{-1} \theta }$ with $r\geq 0$,
$0 \leq \theta < 2\pi$. There always exists $k\in \{0,1,\ldots,m-1\}$ such that
\[
0 \leq \theta - 2k\pi/ m < 2\pi/m.
\]
If we choose $\tau_i = e^{ - 2k\pi \sqrt{-1} / m }$, then
\[
\mbf{1}^T (\tau_i w_i) = r e^{\sqrt{-1} \theta_1 }, \quad
0 \leq \theta_1 < 2\pi/m.
\]
Therefore, without loss of generality, we can assume
$\mbf{1}^T w_i = r e^{\sqrt{-1} \theta }$, with
$0 \leq \theta < 2\pi/m$, in \reff{A:s*:C}. This means that
($\mbf{Im}$ denotes the imaginary part)
\[
\mbf{Im}(  \mbf{1}^T w_i    ) \geq 0, \quad
\mbf{Im}( e^{ \frac{- 2\pi}{m}  \sqrt{-1}  }    \mbf{1}^T w_i    ) \leq 0,
\]
which are equivalent to the conditions
\[
\mbf{1}^T v_i \geq 0, \quad
\sin(2\pi/m)\mbf{1}^T u_i - \cos(2\pi/m) \mbf{1}^T v_i  \geq 0.
\]
Then, the lemma follows from \reff{A:s*:C}.
\end{proof}

A complex vector in $\cpx^n$ can be represented by a
$2n$-dimensional real vector. Let $x = ( x^{re}, x^{im} )$ with
\[
x^{re} =(x_1,\ldots, x_n), \quad x^{im} = (x_{n+1}, \ldots, x_{2n}).
\]
Denote the set
\be \label{set:Sc}
S^c := \left\{ x=(x^{re}, x^{im})
\left|
\baray{c}
\| x^{re} \|^2 + \| x^{im} \|^2 =1, \,
\, x^{re}, x^{im} \in \re^n, \\
\mbf{1}^T x^{im} \geq 0, \,
\sin( \frac{2\pi}{m} )\mbf{1}^T x^{re}
-\cos(\frac{2\pi}{m}) \mbf{1}^T x^{im}  \geq 0
\earay
\right. \right\}.
\ee
For the decomposition of $\mA$ as in \reff{cnn:F=ui+vi},
the weighted Dirac masure
\[
\mu :=  \lmd_1 \dt_{(u_1,v_1)} + \cdots +  \lmd_r \dt_{(u_r,v_r)}
\]
belongs to $\mathscr{B}(S^c)$, the set of Borel measures supported on $S^c$.
It satisfies
\be \label{F=int:dmu:C}
 \mA = \int (x^{re} + \sqrt{-1} x^{im}  )^{\otimes m} \mt{d} \mu.
\ee
Note that
$
\lmd_1 + \cdots + \lmd_r = \int 1 \mt{d} \mu.
$
Conversely, for every $\mu \in \mathscr{B}(S^c)$ satisfying \reff{F=int:dmu:C},
we can always get a decomposition of $\mA$ as in \reff{cnn:F=ui+vi}.
This can be implied by \cite[Prop.~3.3]{ATKMP}.
By Lemma~\ref{lm:opt:cnuc}, $\| \mA \|_{\ast,\cpx}$
equals the optimal value of
\be  \label{cnn:F=int:xom:C}
\left\{\baray{rl}
\min  &  \int 1 \mt{d} \mu \\
s.t. &  \mA = \int (x^{re} + \sqrt{-1} x^{im}  )^{\otimes m} \mt{d} \mu,  \\
&  \mu \in \mathscr{B}(S^c).
\earay\right.
\ee
Note that $S^c = \{ x \in \re^{2n}: \,
h(x) = 0, g_1(x) \geq 0, g_2(x) \geq 0 \}$
where
\be \label{df:h:g:C}
h := x^Tx -1, \,\, g_1: = \mbf{1}^T x^{im}, \,\,
g_2 :=\sin( \frac{2\pi}{m} )\mbf{1}^T x^{im} - \cos( \frac{2\pi}{m} ) \mbf{1}^T x^{re}.
\ee
Let $\mbf{a}^{re}, \mbf{a}^{im} \in \re^{ \N^n_{ \{m\} } }$ be the real vectors such that
\be \label{F=fre+fim}
\mbf{a}^{re}_\af + \sqrt{-1} \mbf{a}^{im}_\af =
\mA_\af \quad \mbox{ if } \quad
x^\af = x_{i_1}\cdots x_{i_m}.
\ee
For each $\af = (\af_1, \ldots, \af_n) \in \N^{n}_{ \{m\} }$,
expand the product
\be \label{Raf:Taf}
(x_1+\sqrt{-1}\, x_{n+1})^{\af_1} \cdots (x_n + \sqrt{-1}\,x_{2n})^{\af_n}
=R_\af(x) + \sqrt{-1}\, T_\af(x),
\ee
for real polynomials
$R_\af, T_\af \in \re[x]:=\re[x_1,\ldots,x_{2n}]$. Then,
\[
\int (x_1+\sqrt{-1}\, x_{n+1})^{\af_1} \cdots (x_n + \sqrt{-1}\,x_{2n})^{\af_n}
\mt{d} \mu
\]
\[
=\int R_\af(x) \mt{d} \mu  + \sqrt{-1}\int T_\af(x) \mt{d} \mu.
\]
Hence, \reff{cnn:F=int:xom:C} is equivalent to
\be \label{cnnopt:F=R+T:mu(C)}
\left\{ \baray{rl}
  \min  &  \int 1 \mt{d} \mu \\
s.t. &  \mbf{a}^{re}_\af = \int R_\af(x) \mt{d} \mu \,\,
(\af \in \N^{n}_{ \{m\} }),  \\
& \mbf{a}^{im}_\af = \int T_\af(x) \mt{d} \mu \,\,
(\af \in \N^{n}_{ \{m\} }),  \\
& \mu \in \mathscr{B}(S^c).
\earay \right.
\ee
To solve \reff{cnnopt:F=R+T:mu(C)}, we can replace
$\mu$ by the vector of its moments. Denote the moment cone
\be \label{scrR(C):0+m}
\mathscr{R}^{c}_{ \{0,m\} }  := \left\{
y \in \re^{ \N^{2n}_{ \{0,m\} } }
\left| \baray{c}
 \exists \mu \in \mathscr{B}(S^c) \, \mbox{ such that }   \\
y_\bt = \int x^\bt \mt{d} \mu\,\,\mbox{ for } \, \bt \in \N^{2n}_{ \{0,m\} }
\earay \right.
\right\}.
\ee
So, \reff{cnnopt:F=R+T:mu(C)} is equivalent to the optimization problem
\be \label{cnn:F=R+T:y(C)}
\left\{ \baray{rl}
 \min  &  (y)_0 \\
s.t. &  \langle R_\af, y \rangle =  \mbf{a}^{re}_\af \,\,
(\af \in \N^{n}_{ \{m\} }),  \\
& \langle T_\af, y \rangle = \mbf{a}^{im}_\af \,\,
(\af \in \N^{n}_{ \{m\} }),  \\
& y \in \mathscr{R}^{c}_{ \{0, m\} }.
\earay \right.
\ee

\subsection{An algorithm}

The cone $\mathscr{R}^{c}_{ \{0, m\} }$
can be approximated by semidefnite relaxations.
For $h, g_1, g_2$ as in \reff{df:h:g:C}, denote the cones
\begin{align}
\label{scr(S):C:2k}
\mathscr{S}^{c, 2k} & := \left\{
z \in \re^{ \N^{2n}_{ [0,2k] } }
\left| \baray{c}
M_k(z) \succeq 0,  L^{(k)}_{h}(z) = 0, \\
 L^{(k)}_{g_1}(z) \succeq 0,  L^{(k)}_{g_2}(z) \succeq 0
\earay \right.
\right\}, \\
\label{scr(S):C:2k:0+m}
\mathscr{S}^{c,2k}_{ \{0,m\} } & := \left\{
y \in \left. \re^{ \N^{2n}_{ \{0,m\} } } \right|
\exists \, z \in  \mathscr{S}^{c,2k}, \,  \, y = z|_{ \{0,m \} }
\right\}.
\end{align}
Clearly, $\mathscr{S}^{c,2k}_{ \{0,m\} }$ is a projection of
$\mathscr{S}^{c,2k}$. For all $k\geq m/2$, we have
\[
\mathscr{R}^{c}_{ \{0,m\} }  \subseteq
\mathscr{S}^{c,2k+2}_{ \{0,m\} } \subseteq \mathscr{S}^{c,2k}_{ \{0,m\} }.
\]
Indeed, it holds that (cf.~\cite[Prop.~3.3]{LMOPT})
\be  \label{sdr:R:Sc}
\mathscr{R}^{c}_{ \{0,m\} } = \bigcap_{k \geq m/2 }
\mathscr{S}^{c, 2k}_{ \{0,m\} }.
\ee
This produces the hierarchy of semidefinite relaxations
\be \label{cnn:mom:z(C)}
\left\{ \baray{rl}
\| \mA \|_{k\ast,\cpx}  \, := \, \min  &  (z)_0 \\
s.t. &  \langle R_\af, z \rangle =  \mbf{a}^{re}_\af \,\,
(\af \in \N^{n}_{ \{m\} }),  \\
& \langle T_\af, z \rangle = \mbf{a}^{im}_\af \,\,
(\af \in \N^{n}_{ \{m\} }),  \\
& z \in \mathscr{S}^{c,2k},
\earay \right.
\ee
for $k=m_0, m_0+1, \ldots$ ($m_0=\lceil m/2 \rceil$).
Like \reff{rhok:mcr:evm}, we also have
\be  \label{rl:rhok:cpx}
\| \mA \|_{m_0\ast,\cpx} \leq \cdots \leq
\| \mA \|_{k\ast,\cpx} \leq \cdots \leq \| \mA \|_{\ast,\cpx}.
\ee

\begin{alg} \label{alg:cnn:cpx}
For a given tensor $\mA \in \mt{S}^m(\cpx^n)$,
let $k = m_0$ and do:
\bit

\item [Step 1] Solve the semidefinite relaxation \reff{cnn:mom:z(C)},
for an optimizer $z^{k}$.

\item [Step 2]  Let $y^{k} := z^{k}\big|_{ \{0,m\} }$
(see \reff{trun:z|0,m} for the truncation).
Check whether or not
$y^{k} \in \mathscr{R}^{c}_{ \{0,m\} }$.
If yes, then
$\| \mA \|_{\ast, \re} = \| \mA \|_{k\ast,\cpx}$ and go to Step~3;
otherwise, let $k :=k+1$ and go to Step~1.

\item [Step 3] Compute the decompositions of $y^{k}$ as
\[
y^{k} =  \lmd_1 [(u_1,v_1)]_{0,m} + \cdots + \lmd_r [(u_r,v_r)]_{0,m}
\]
with all $\lmd_i >0, (u_i,v_i) \in S^c$.
This gives the nuclear decomposition
\[
\mA =   \lmd_1 (u_1 + \sqrt{-1}\, v_1)^{\otimes m}  + \cdots +
\lmd_r (u_r + \sqrt{-1}\, v_r)^{\otimes m}
\]
such that
$\sum_{i=1}^{r} \lmd_i  = \| \mA \|_{\ast,\cpx}$.

\eit

\end{alg}

In the above, the method in \cite{ATKMP}
can be used to check if $y^{k} \in \mathscr{R}^c_{ \{0,m\} }$ or not.
If yes, we can also get a nuclear decomposition.
It requires to solve a moment optimization problem
whose objective is randomly generated.

\subsection{Convergence properties}

Denote the real polynomial vectors:
\[
R(x):= (R_\af(x))_{ \af \in \N^{n}_{ \{m\} } }, \quad
T(x):= (T_\af(x) )_{ \af \in \N^{n}_{ \{m\} } },
\]
where $R_\af, T_\af$ are as in \reff{Raf:Taf}.
Their length $D=\binom{2n+m-1}{m}$. Denote
\be
\mathscr{P}(S^c)_{0,m} \, := \, \{ t + p  \, \mid \,
t \in \re, \, p\in \re[x]_m^{hom}, \, t+p \geq 0 \mbox{ on } \, S^c
\}.
\ee
The cones $\mathscr{R}^c_{\{0,m\}}$ and
$\mathscr{P}(S^c)_{0,m}$ are dual to each other \cite{LMOPT},
so the dual optimization problem of \reff{cnn:F=R+T:y(C)} is
\be \label{mx<p,f>:p1p2:P(C)}
\left\{ \baray{rl}
 \max\limits_{ p_1, p_2 \in \re^D } &
 p_1^T \mbf{a}^{re}    +  p_2^T \mbf{a}^{im}  \\
s.t. \quad &  1 - p_1^T R(x) - p_2^T T(x) \in \mathscr{P}(S^c)_{0,m}.
\earay \right.
\ee

\begin{lemma} \label{monng:dual:Cpx}
Let $\mbf{a}^{re}, \mbf{a}^{im}$ be as in \reff{F=fre+fim}.
Then, both \reff{cnn:F=R+T:y(C)} and \reff{mx<p,f>:p1p2:P(C)}
achieve the same optimal value which equals
$\| \mA \|_{\ast,\cpx}$.
\end{lemma}
\begin{proof}
The origin is an interior point of \reff{mx<p,f>:p1p2:P(C)}.
So, \reff{cnn:F=R+T:y(C)} and \reff{mx<p,f>:p1p2:P(C)}
have the same optimal value,
and \reff{cnn:F=R+T:y(C)} achieves it (cf.~\cite[\S2.4]{BTN}).
In the next, we prove that \reff{mx<p,f>:p1p2:P(C)}
also achieves its optimal value. Let
\[
w = x^{re} + \sqrt{-1} x^{im}, \quad
q_\af = (p_1)_\af - \sqrt{-1} (p_2)_\af.
\]
\[
\baray{rcl}
q(w) &=&  \sum_{ |\af| =m}  q_\af
\big( R_\af(x) + \sqrt{-1} T_\af(x) \big) \\
 & = &  \sum_{ |\af| =m}  q_\af  w^\af .
\earay
\]
Clearly, $q(w)$ is a form of degree $m$ and
in $w \in \cpx^n$, and
($\mbf{Re}$ denotes the real part)
\[
\mbf{Re} \,\, q(w) = p_1^T R(x) + p_2^T T(x).
\]
Let
$
B=\{ x^{re} + \sqrt{-1} x^{im}: \, (x^{re}, x^{im} ) \in S^c \},
$
which is a subset of the complex unit sphere $\| w \| =1$.
When $(p_1, p_2)$ is feasible for \reff{mx<p,f>:p1p2:P(C)},
the polynomial
\[
p(x) \, := \, 1 - p_1^T R(x) - p_2^T T(x)
\quad \geq 0 \quad \mbox{ on } S^c.
\]
So,
$
\mbf{Re}\,\, q(w) \leq 1
$
for all $w \in B.$
For all $w \in \cpx^n$ with $\|w\|=1$, there exist
$\tau^m=1$ and $a \in B$ such that $w=\tau a$.
This is shown in the proof of Lemma~\ref{lm:opt:cnuc}, so
\[
\mbf{Re}\,\, q(w)   = \mbf{Re}\,\, q(\tau a)
=  \mbf{Re}\,\, q( a)  \leq 1.
\]
The above is true for all unit complex vectors $w$, hence
\[
\mbf{Re}\,\, q(w) \leq 1 \quad \forall w \in \cpx^n: \, \|w\|=1.
\]
Because $q(w)$ is homogeneous in $w$, the above implies that
\[
|q(w)| \leq 1 \quad \forall w \in \cpx^n: \, \|w\|=1.
\]
So, there exists $M >0$ such that
$ \|vec(q) \| \leq M $
for all $q$ satisfying the above.
Since $\| vec(q) \|^2 = \| vec(p_1) \|^2 + \| vec(p_2) \|^2$,
the feasible set of \reff{mx<p,f>:p1p2:P(C)} is compact.
So, \reff{mx<p,f>:p1p2:P(C)} must achieve its optimal value.
\end{proof}

Next, we study the properties of the relaxation \reff{cnn:mom:z(C)}.
For $h,g:=(g_1,g_2)$ as in \reff{df:h:g:C}, denote the cones of polynomials
\be \label{Qkc:cpx}
Q_k^c \, := \, \mbox{Ideal}(h)_{2k} + \mbox{Qmod}_{2k}(g),
\quad Q^c \, := \, \bigcup_{k\geq 1} Q_k^c.
\ee
The cones $Q_k^c$ and $\mathscr{S}^{c,2k}$
are dual to each other \cite{LMOPT},
so the dual optimization problem of \reff{cnn:mom:z(C)} is
\be \label{mx<p1+p2,f>:SOS:Q(C)}
\left\{ \baray{rl}
 \max\limits_{ p_1, p_2 \in \re^D } &
 p_1^T \mbf{a}^{re}    +  p_2^T \mbf{a}^{im}  \\
s.t. &  1 - p_1^T R(x) - p_2^T T(x) \in Q_k^c.
\earay \right.
\ee

\begin{lemma} \label{lm:sos-mom:QkSc}
Let $\mbf{a}^{re}, \mbf{a}^{im}$ be the real vectors as in \reff{F=fre+fim}.
Then, for each $k\geq m_0$, both \reff{cnn:mom:z(C)} and \reff{mx<p1+p2,f>:SOS:Q(C)}
achieve the same optimal value which equals $\| \mA \|_{k\ast,\cpx}$.
Moreover, $\| \mA \|_{k\ast,\cpx}$
is a norm function in $\mA \in \mt{S}^m(\cpx^n)$.
\end{lemma}
\begin{proof}
The proof is almost the same as for Lemmas~\ref{ach:opv:k:om}
and \ref{achval:Qk:evm}.
For each $k\geq m_0$, the origin is an interior point of
\reff{mx<p1+p2,f>:SOS:Q(C)}.
The vanishing ideal of $S^c$ is $\mbox{Ideal}(h)$,
so the set $Q_k^c$ is closed, implied by Theorem~3.35 of \cite{Laurent}
or Theorem~3.1 of \cite{Marsh03}.
In the proof of Lemma~\ref{monng:dual:Cpx},
we showed that the feasible set of \reff{mx<p,f>:p1p2:P(C)}
is compact. Since $Q_k^c \subseteq \mathscr{P}(S^c)_{\{0,m\}}$,
the feasible set of \reff{mx<p1+p2,f>:SOS:Q(C)} is also compact.
By the linear conic duality theory \cite[\S2.4]{BTN},
both \reff{cnn:mom:z(C)} and \reff{mx<p1+p2,f>:SOS:Q(C)}
achieve the same optimal value.
We can similarly prove that
$\| \mA \|_{k\ast,\cpx}$ is a norm function in $\mA$.
We omit the proof here, since it is almost the same as for
Lemmas~\ref{ach:opv:k:om} and \ref{achval:Qk:evm}.
\end{proof}

The convergence properties of Algorithm~\ref{alg:cnn:cpx} are as follows.

\begin{theorem}  \label{thm:cvg:cpx}
Let $\| \mA \|_{k\ast,\cpx}$ be the optimal value of \reff{cnn:mom:z(C)}.
For all $\mA \in \mt{S}^m(\cpx^n)$,
Algorithm~\ref{alg:cnn:cpx} has the following properties:
\bit
\item [(i)]
$\lim\limits_{k\to\infty} \| \mA \|_{k\ast,\cpx}
= \| \mA \|_{\ast,\cpx}$.

\item [(ii)] Let $(p_1^*, p_2^*)$ be an optimal pair for
\reff{mx<p,f>:p1p2:P(C)}. If
\[
1-(p_1^*)^TR(x) -(p_2^*)^T T(x) \in Q^c,
\]
then $\| \mA \|_{k\ast,\cpx} = \| \A \|_{\ast, \cpx}$
for all $k$ sufficiently big.

\item [(iii)] If $y^{k} \in \mathscr{R}^{c}_{ \{0,m\} }$ for some order $k$,
then $\| \mA \|_{k\ast,\cpx} = \| \A \|_{\ast,\cpx}$.

\item [(iv)] The sequence $\{ y^{k} \}_{k=m_0}^\infty$ is bounded,
and each of its accumulation points belongs to $\mathscr{R}^{c}_{ \{0,m\} }$.
Moreover, if the nuclear decomposition of $\mA$ over $\cpx$ is unique,
then $y^{k}$ converges to a point in
$\mathscr{R}^{c}_{ \{0,m\} }$ as $k \to \infty$.

\eit
\end{theorem}
\begin{proof}
(i) By Lemma~\ref{monng:dual:Cpx},
for every $\eps>0$, there exist $s_1, s_2 \in \re^D$ such that
\[
1 - (s_1)^TR -(s_2)^T T>0  \mbox{ on }  S^c, \quad
\langle s_1, \mbf{a}^{re} \rangle +  \langle s_2, \mbf{a}^{im} \rangle
\geq \| \A \|_{\ast,\cpx} - \eps.
\]
By Theorem~\ref{thm:Put}, there exists $k_1$ such that
\[
1 - (s_1)^TR -(s_2)^T T  \in Q_{k_1}^c.
\]
By Lemma~\ref{lm:sos-mom:QkSc}, we can get
$
\| \mA \|_{k_1\ast,\cpx}  \geq  \| \A \|_{\ast,\cpx} - \eps.
$
The monotonicity relation \reff{rl:rhok:cpx} and the above imply that
\[
\| \mA \|_{\ast,\cpx} \geq \lim\limits_{k\to\infty}  \| \mA \|_{k\ast,\cpx}
\geq \| \mA \|_{\ast,\cpx} - \eps.
\]
Since $\eps>0$ can be arbitrarily small,
the item (i) follows directly.

(ii) If $1-(p_1^*)^T R -(p_2^*)^T T \in Q^c$, then
$1-(p_1^*)^T R -(p_2^*)^T T \in Q_{k_2}^c$ for some $k_2 \in \N$.
By Lemma~\ref{monng:dual:Cpx}, we know that
\[
\| \mA \|_{\ast,\cpx} = \langle p_1^*, \mbf{a}^{re} \rangle +
\langle p_2^*, \mbf{a}^{im} \rangle  \leq \| \mA \|_{k_2\ast,\cpx}.
\]
Then, \reff{rl:rhok:cpx} implies that
$\| \mA \|_{k\ast,\cpx} = \| \mA \|_{\ast,\cpx}$
for all $k \geq k_2$.

(iii) If $y^k \in \mathscr{R}^c_{ \{0,m\} }$ for some order $k$,
then $\| \mA \|_{k\ast,\cpx} \geq \| \mA \|_{\ast,\cpx}$,
by Lemma~\ref{monng:dual:Cpx}.
The equality
$\| \mA \|_{k\ast,\cpx} = \| \mA \|_{\ast,\cpx}$
follows from \reff{rl:rhok:cpx}.

(iv) Note that $ (z^k)_0 = (y^k)_0 = \| \mA \|_{k\ast,\cpx}$ for all $k$.
The condition $L_{1-\|x\|^2}^{(k)}(z^k)=0$ implies that
\[
(z^k)_{2e_1 + 2\bt} + \cdots (z^k)_{2e_{2n} + 2\bt}  = (z^k)_{2\bt}
\]
for all $\bt \in \N^{2n}_{[0,2k]}$.
By induction, one can easily show that
\[
(z^k)_{2\bt} \leq (z^k)_0 \quad \forall \, \bt \in \N^{2n}_{[0,2k]}.
\]
Since $y^k$ is a truncation of $z^k$
and $M_k( z^k ) \succeq 0$, we get
\[
|(y^k)_{\bt}|^2 = |(z^k)_{\bt}|^2  \leq (z^k)_{2\bt}  (z^k)_{0}
\leq |(z^k)_{0} |^2 = |(y^k)_{0} |^2 = \| \mA \|_{\ast, \cpx}^2.
\]
This shows that the sequence $\{ y^k \}$ is bounded.
For all $k\geq m/2$, it holds that
$y^k \in \mathscr{S}^{c,2k}_{ \{0,m\} }$.
The distance between $\mathscr{S}^{c,2k}_{ \{0,m\} }$ and
$\mathscr{R}^c_{ \{0,m\} }$ tends to zero as $k\to \infty$
(cf.~\cite[Prop.~3.4]{LMOPT}).
Therefore, every accumulation point $\hat{y}$ of the sequence $\{ y^k \}$
belongs to $\mathscr{R}^{c}_{ \{0,m\} }$.
This can also be implied by \reff{sdr:R:Sc}. So,
$
\hat{y} = \sum_{i=1}^r \lmd_i [(u_i, v_i)]_{0,m},
$
with $\lmd_i >0$ and $(u_i, v_i) \in S^c$.
The feasibility condition in \reff{cnn:mom:z(C)}
and the relation \reff{Raf:Taf} imply that
\[
\mA = \sum_{i=1}^r \lmd_i (u_i + \sqrt{-1} v_i)^{\otimes m}.
\]
When the nuclear decomposition of $\mA$ is unique,
$\lmd_i$ and $(u_i, v_i) \in S^c$
are also uniquely determined.
So, the accumulation point $\hat{y}$ is unique and
$y^k$ converges to a point in $\mathscr{R}^{c}_{ \{0,m\} }$
as $k \to \infty$.
\end{proof}

\section{Numerical examples}
\label{sc:num}

This section presents numerical experiments
for nuclear norms of symmetric tensors.
The computation is implemented in MATLAB R2012a,
on a Lenovo Laptop with CPU@2.90GHz and RAM 16.0G.
Algorithm~\ref{alg:R:odd} is applied for
real nuclear norms of real odd order tensors,
Algorithm~\ref{alg:even:m} is for
real nuclear norms of real even order tensors,
while Algorithm~\ref{alg:cnn:cpx} is for
complex nuclear norms of all tensors.
These algorithms can be implemented in software
{\tt Gloptipoly~3} \cite{Gloptipoly}
by calling the semidefinite program package {\tt SeDuMi} \cite{sedumi}.

Since our methods are numerical, we display only
four decimal digits for the computational results.
For a nuclear decomposition
$\mA = (u_1)^{\otimes m} + \cdots + (u_r)^{\otimes m}$,
we display it by listing the vectors
$u_1, \ldots, u_r$
column by column, from the left to right.
If one row block is not enough, we continue the display
in the bottom, separated by one blank row.

Recall that $e$ is the vector of all ones,
and $e_i$ denotes the $i$th standard unit vector
(i.e., the vector whose $i$th entry is one
and all others are zeros).
We begin with some tensor examples from
Friedland and Lim~\cite{FriLim14b}.

\begin{exm}
(\cite{FriLim14b})
(i) Consider the tensor in $\mt{S}^3(\re^2)$ such that
\[
\mA = \frac{1}{\sqrt{3}}\big(
e_1 \otimes e_1 \otimes e_2 + e_1 \otimes e_2 \otimes e_1 +
e_2 \otimes e_1 \otimes e_1
\big).
\]
We got
$ \| \mA \|_{\ast, \re} = \| \mA \|_{2\ast, \re} = \sqrt{3}$ and
$ \| \mA \|_{\ast, \cpx} = \| \mA \|_{2\ast, \cpx} = 3/2$.
It took about $1$ second.
The real nuclear decomposition
$\mA = \sum_{i=1}^3 (u_i)^{\otimes 3}$ is {\tiny
\begin{verbatim}
    0.0000   -0.7937    0.7937
   -0.5774    0.4582    0.4582
\end{verbatim} \noindent}and
the complex nuclear decomposition
$\mA = \sum_{i=1}^3 (w_i)^{\otimes 3}$ is {\tiny
\begin{verbatim}
  -0.5873 + 0.2740i   0.4582 - 0.4583i   0.6456 - 0.0566i
   0.2944 + 0.3511i  -0.0002 + 0.4582i   0.4513 + 0.0797i
\end{verbatim}}

\noindent
(ii) Consider the tensor in $\mt{S}^3(\re^2)$ such that
\[
\mA = \frac{1}{2}\big(
e_1 \otimes e_1 \otimes e_2 + e_1 \otimes e_2 \otimes e_1 +
e_2 \otimes e_1 \otimes e_1 - e_2 \otimes e_2 \otimes e_2
\big).
\]
We got $\| \mA \|_{\ast, \re} = \| \mA \|_{2\ast, \re} = 2$
and $\| \mA \|_{\ast, \cpx} = \| \mA \|_{2\ast, \cpx} = \sqrt{2}$.
It took about $1$ second.
The real nuclear decomposition
$\mA = \sum_{i=1}^3 (u_i)^{\otimes 3}$ is {\tiny
\begin{verbatim}
    0.0000   -0.7565    0.7565
   -0.8736    0.4368    0.4368
\end{verbatim} \noindent}while
the complex nuclear decomposition
$\mA = \sum_{i=1}^2 (w_i)^{\otimes 3}$ is{\tiny
\begin{verbatim}
   0.5456 - 0.3150i  -0.5456 + 0.3150i
   0.3150 + 0.5456i   0.3150 + 0.5456i
\end{verbatim} \noindent}The
nuclear norms are the same as in \cite{FriLim14b}.
\qed
\end{exm}

Next, we see some tensors of order four.

\begin{exm}
(i) Consider the tensor $\mA \in \mt{S}^4(\re^3)$ such that
\[
\mA = e^{\otimes 4} - e_1^{\otimes 4}  - e_2^{\otimes 4} - e_3^{\otimes 4}.
\]
We got $\| \mA \|_{\ast, \re} = \| \mA \|_{2\ast, \re} =12$
and $\| \mA \|_{\ast, \cpx} =  \| \mA \|_{3\ast, \cpx} \approx 11.8960$.
It took about $15$ seconds.
The real nuclear decomposition is the same as above.
The complex nuclear decomposition
$\mA = \sum_{i=1}^{9} (w_i)^{\otimes 4}$ is {\tiny
\begin{verbatim}
  -0.1152 + 0.0714i   0.0332 - 0.1001i   0.0479 - 0.1059i  -0.1102 + 0.0539i  -0.0845 + 0.8376i
   0.5316 + 0.5596i   0.6145 + 0.5968i   0.0519 - 0.1094i  -0.1068 + 0.0498i   0.0823 + 0.8373i
  -0.1186 + 0.0752i   0.0288 - 0.0968i   0.5905 + 0.5684i   0.5654 + 0.5880i   0.0022 + 0.8307i

   0.1285 + 0.7122i   0.0163 + 0.6866i   0.5921 + 0.6105i   0.5648 + 0.5379i
  -0.0124 + 0.6932i  -0.1320 + 0.7066i  -0.1017 + 0.0364i   0.0674 - 0.1137i
  -0.1158 + 0.7086i   0.1153 + 0.7018i  -0.0984 + 0.0319i   0.0711 - 0.1172i
\end{verbatim}
}

\noindent
(ii) Consider the tensor $\mA \in \mt{S}^4(\re^3)$ such that
\[
\mA  =   (e_1 + e_2 )^{\otimes 4}  +
(e_1 +  e_3)^{\otimes 4}  - ( e_2 + e_3)^{\otimes 4}.
\]
We got $\| \mA \|_{\ast, \re} = \| \mA \|_{2\ast, \re} =
\| \mA \|_{\ast, \cpx} = \| \mA \|_{2\ast, \cpx} = 12$.
It took about $12$ seconds.
The real and complex nuclear decompositions are the same as above.

\noindent
(iii) Consider the tensor $\mA \in \mt{S}^4(\re^3)$ such that
\[
\mA  =    (e_1 + e_2 - e_3)^{\otimes 4}  +
(e_1 - e_2 + e_3)^{\otimes 4}  + (-e_1 + e_2 + e_3)^{\otimes 4}
- (e)^{\otimes 4}.
\]
We got $\| \mA \|_{\ast, \re} = \| \mA \|_{2\ast, \re} =
\| \mA \|_{\ast, \cpx} = \| \mA \|_{2\ast, \cpx} = 36$.
It took about $7$ seconds.
The real and complex nuclear decompositions are the same as above.
\qed
\end{exm}

The following are some examples of complex-valued tensors.

\begin{exm}
(i) Consider the tensor $\mA \in \mt{S}^3(\cpx^3)$ such that
\[
\mA_{i_1 i_2 i_3} = \sqrt{-1}^{i_1 i_2 i_3 }.
\]
We got $\| \mA \|_{\ast, \cpx} =  \| \mA \|_{2\ast, \cpx} \approx 8.8759$.
It took about $1$ second.
The nuclear decomposition $\mA = \sum_{i=1}^5 (w_i)^{\otimes 3}$ is {\tiny
\begin{verbatim}
   0.6024 + 0.3478i  -0.6019 + 0.3475i   0.0000 - 0.6894i   0.6262 + 0.3615i  -0.0000 + 0.7230i
   0.0000 - 0.0000i   0.0000 - 0.0000i   0.0000 - 0.0000i  -0.7913 + 0.6476i   0.9565 - 0.3615i
  -0.6024 - 0.3478i   0.6019 - 0.3475i  -0.0000 + 0.6894i   0.6262 + 0.3615i   0.0000 + 0.7230i
\end{verbatim}}

\noindent
(ii) Consider the tensor $\mA \in \mt{S}^4(\cpx^3)$ such that
\[
\mA_{i_1 i_2 i_3 i_4} = ( \sqrt{-1} )^{i_1}+
( -1 )^{i_2}+  ( -\sqrt{-1} )^{i_3}+   ( 1 )^{i_4};
\]
We got $\| \mA \|_{\ast, \cpx} =
\| \mA \|_{3\ast, \cpx} \approx  26.9569$.
It took about $17$ seconds.
The nuclear decomposition $\mA = \sum_{i=1}^7 (w_i)^{\otimes 4}$
is {\tiny
\begin{verbatim}
   0.6274 + 0.5703i  -0.7090 + 0.2840i   0.2324 - 1.0695i  -0.1256 + 0.3158i
   0.6275 + 0.5699i   0.5471 + 0.1352i   0.2938 + 0.5936i   0.2671 + 0.1962i
  -0.5940 - 0.6378i   0.5286 + 0.1184i   0.2875 + 0.6043i   0.1922 + 0.2772i

  -0.1432 + 0.8472i   0.0955 + 0.9276i   0.7074 + 0.9184i
  -0.1440 + 0.8475i   0.3732 + 0.7080i  -0.0292 + 0.6956i
   0.9490 + 0.2131i   0.4154 + 0.7460i  -0.0031 + 0.6971i
\end{verbatim}
}

\noindent
(iii) Consider the tensor $\mA \in \mt{S}^5(\cpx^3)$ such that
\[
\mA_{i_1 i_2 i_3 i_4 i_5} =  (\sqrt{-1})^{ i_1 i_2 i_3 i_4 i_5} +
(-\sqrt{-1})^{ i_1 i_2 i_3 i_4 i_5} .
\]
We got $\| \mA \|_{\ast, \cpx} =
\| \mA \|_{2\ast, \cpx} \approx 49.5626$.
It took about $4.7$ seconds.
The nuclear decomposition
$\mA = \sum_{i=1}^6 (w_i)^{\otimes 5}$ is {\tiny
\begin{verbatim}
   0.2711 + 0.8335i   0.8651 + 0.0003i   0.6741 + 0.6909i
  -0.2255 - 0.6963i  -0.7224 + 0.0008i   0.7712 - 0.6030i
   0.2711 + 0.8335i   0.8651 + 0.0003i   0.6741 + 0.6909i

   0.5542 + 0.0006i   0.8654 + 0.4276i   0.1743 + 0.5348i
   1.1499 - 0.0004i  -0.3352 + 0.9198i   0.3603 + 1.1102i
   0.5542 + 0.0006i   0.8654 + 0.4276i   0.1743 + 0.5348i
\end{verbatim}
}

\noindent
(iv) Consider the tensor $\mA \in \mt{S}^6(\cpx^3)$ such that
\[
\mA_{i_1 i_2 i_3 i_4 i_5 i_6} =
( 1 + \sqrt{-1} )^{i_1+\cdots+i_6-6} + (1-\sqrt{-1})^{i_1+\cdots+i_6-6}.
\]
We got $\| \mA \|_{\ast, \cpx} =
\| \mA \|_{2\ast, \cpx} = 686 $.
It took about $4.8$ seconds.
The nuclear decomposition is
\[
\mA = \bbm 1  \\ 1 - \sqrt{-1} \\  - 2\sqrt{-1} \ebm^{\otimes 6} +
 \bbm 1  \\ 1 + \sqrt{-1} \\  2\sqrt{-1} \ebm^{\otimes 6}.
\]
\qed
\end{exm}

\begin{exm} \label{exm:A=i1+i2+i3}
Consider the tensor $\mA \in \mt{S}^3(\re^n)$ such that
\[
\mA_{i_1 i_2 i_3} = i_1 + i_2 + i_3.
\]
For a range of values of $n$,
the real and complex nuclear norms
$\| \mA \|_{\ast, \re}$ and $\| \mA \|_{\ast,\cpx}$
are reported in Table~\ref{tab:A=i1+i2+i3}.
We list the order $k$ for which
$\| \mA \|_{\ast,\F} = \| \mA \|_{k\ast,\F}$
and the length of the nuclear decomposition,
as well as the consumed time (in seconds).
\begin{table}[htb]
\caption{Nuclear norms of the tensor in Example~\ref{exm:A=i1+i2+i3}.}
\label{tab:A=i1+i2+i3}
\btab{|c|c|r|c|c|r|} \hline
$n$  & $\F$ & $\| \mA \|_{\ast, \F}$ &  $k$ & {\tt length} & {\tt time}  \\ \hline
$2$  &  $\re$  &  $13.4164$  &  $2$  &  $3$ &   0.81       \\  \hline
$2$  &  $\cpx$  & $13.2114$   &  $2$  & $3$  &  1.31      \\  \hline
$3$  &  $\re$  &  $33.6749$  &  $2$  & $3$ &    0.90     \\ \hline
$3$  &  $\cpx$  & $32.9505$   &  $2$  & $3$  &  1.92      \\ \hline
$4$  &  $\re$  &  $65.7267$  &  $2$  & $3$  &  0.93        \\ \hline
$4$  &  $\cpx$  & $64.0886$ &  $2$  & $3$  &   3.73    \\ \hline
$5$  &  $\re$  &  $111.2430$  &  $2$  & $3$  &  0.97      \\ \hline
$6$  &  $\re$  &  $171.7091$  &  $2$  & $3$  &  1.08    \\ \hline
$7$  &  $\re$  &  $248.4754$  &  $2$  & $3$  &  1.23       \\ \hline
$8$  &  $\re$  &  $342.7886$  &  $2$  & $3$  &  1.67       \\ \hline
$9$  &  $\re$  &  $455.8125$  &  $2$  & $3$  &  2.45      \\ \hline
$10$ &  $\re$  &  $588.6425$  &  $2$  & $3$  &  2.66      \\ \hline
\etab
\end{table}
For neatness, we only display nuclear decompositions
for $n=3$. The real nuclear decomposition
$\mA = \sum_{i=1}^3 (u_i)^{\otimes 3}$ is {\tiny
\begin{verbatim}
    0.3689   -0.8633    1.5317
   -0.1899   -0.3318    1.8215
   -0.7487    0.1996    2.1113
\end{verbatim} \noindent}while
the complex nuclear decomposition
$\mA = \sum_{i=1}^3 (w_i)^{\otimes 3}$ is {\tiny
\begin{verbatim}
   0.0851 + 0.6890i   1.4795 - 0.0001i   0.5552 + 0.4168i
  -0.2504 + 0.5318i   1.7763 + 0.0001i   0.5878 + 0.0477i
  -0.5858 + 0.3745i   2.0730 + 0.0004i   0.6203 - 0.3214i
\end{verbatim}}
\qed
\end{exm}

\begin{exm} \label{exm:cos(1/i1+c+1/i4)}
Consider the tensor $\mA \in \mt{S}^4(\re^n)$ such that
\[
\mA_{i_1 i_2 i_3 i_4} =
\cos \left(  \frac{1}{i_1}+\frac{1}{i_2}+\frac{1}{i_3}+\frac{1}{i_4} \right).
\]
The nuclear norms, the order $k$ for which
$\| \mA \|_{\ast, \F} = \| \mA \|_{k\ast, \F}$,
the lengths of the nuclear decompositions,
and the consumed time (in seconds)
are displayed in Table~\ref{tab:cos(1/i1+c+1/i4)}
for a range of values of $n$.
\begin{table}[htb]
\caption{Nuclear norms of the tensor in Example~\ref{exm:cos(1/i1+c+1/i4)}.}
\label{tab:cos(1/i1+c+1/i4)}
\btab{|c|c|r|c|c|r|} \hline
$n$  &  $\F$    &  $\| \mA \|_{\ast, \F}$ &  $k$ & {\tt length} & {\tt time}  \\ \hline
$2$  &  $\re$   &  $4.9001$  &  2  &   4  &  0.93   \\  \hline
$2$  &  $\cpx$  &  $3.9911 $  &  2  &  4  &  2.04    \\  \hline
$3$  &  $\re$   &  $  10.7246$  &  2   &   4  & 1.02     \\ \hline
$3$  &  $\cpx$  &  $ 8.1627 $  &  2   &  4   &  10.05    \\ \hline
$4$  &  $\re$  &  $ 18.0100 $  &  2   &  4  &  1.14     \\ \hline
$4$  &  $\cpx$  &  $13.1108$  &  2   &  4   &   131.13    \\ \hline
$5$  &  $\re$  &  $26.9770 $  &  2   &  4   &  1.33    \\ \hline
$6$  &  $\re$  &  $37.8395 $  &  2   &  4   &  1.86     \\ \hline
$7$  &  $\re$  &  $50.7373 $  &  2   &  4   &  2.78    \\ \hline
$8$  &  $\re$  &  $65.7485 $  &  2  &  4    &  4.75   \\ \hline
$9$  &  $\re$  &  $ 82.9121 $  &  2   &  4   &  9.30    \\ \hline
$10$ &  $\re$  &  $102.2442 $  &  2   &  4   &  21.49    \\ \hline
\etab
\end{table}
For neatness, we only display nuclear decompositions
for $n=3$. The real nuclear decomposition is
$\mA = (u_1)^{\otimes 4} + (u_2)^{\otimes 4}
-(u_3)^{\otimes 4} -(u_4)^{\otimes 4} $,
where $u_1,u_2,u_3,u_4$ are respectively given as {\tiny
\begin{verbatim}
   -0.0261    0.9989   -0.6615    1.0988
    0.7131    0.1816    0.2850    0.9044
    0.9287   -0.1114    0.5965    0.7863
\end{verbatim} \noindent}The
complex nuclear decomposition
$\mA = \sum_{i=1}^4 (w_i)^{\otimes 4}$ is {\tiny
\begin{verbatim}
  -0.0001 - 0.1505i   0.2673 + 0.7374i   0.7395 + 0.2678i   0.6659 + 0.6676i
  -0.0010 + 0.3845i   0.5967 + 0.3270i   0.3287 + 0.5959i   0.5021 + 0.5032i
  -0.0013 + 0.5481i   0.6771 + 0.1666i   0.1680 + 0.6759i   0.4172 + 0.4181i
\end{verbatim}
}
\qed
\end{exm}

For a tensor $\mA \in \mt{S}^m( \cpx^n )$, define
\be \label{df:A(x)}
\mA(x)  := \sum_{i_1,\ldots, i_m = 1}^n
\mA_{i_1 \ldots i_m} \cdot x_{i_1} \cdots x_{i_m}.
\ee
Clearly, $\mA(x)$ is a homogeneous polynomial
in $x:=(x_1,\ldots, x_n)$ and of degree $m$.
There is a bijection between
the symmetric tensor space  $\mt{S}^m( \cpx^n )$
and the space of homogeneous polynomials of degree $m$
(cf.~\cite{GPSTD,OedOtt13}). So, we can equivalently display $\mA$
by showing the polynomial $\mA(x)$.
Moreover, the decomposition $\mA = \sum_{i=1}^r \pm (u_i)^{\otimes m}$
is equivalent to $\mA(x) = \sum_{i=1}^r \pm (u_i^Tx)^m$.
Thus, we can also display a nuclear decomposition
by writing $\mA(x)$ as a sum of power of linear forms.

\begin{exm}\label{exm:poly:A(x)}
(i) Consider the tensor $\mA \in \mt{S}^3(\re^3)$ such that
\[
\mA(x) = x_1x_2x_3.
\]
We got $\| \mA \|_{\ast, \re} = \| \mA \|_{2\ast, \re} =
\| \mA \|_{\ast, \cpx} = \| \mA \|_{2\ast, \cpx} = \sqrt{3}/2$.
The real nuclear decomposition of $\mA$ is given as
%
%
\[
\baray{rl}
\mA(x) = &\frac{1}{24}
\Big( (-x_1-x_2+x_3)^3 + (-x_1+x_2-x_3)^3 \\
& \qquad + (x_1-x_2-x_3)^3 + (x_1+x_2+x_3)^3 \, \Big).
\earay
\]
The above also serves as a complex nuclear decomposition.

\noindent
(ii) Consider the tensor $\mA \in \mt{S}^4(\re^4)$ such that
\[
\mA(x) = x_1x_2x_3x_4.
\]
We got $\| \mA \|_{\ast, \re} = \| \mA \|_{2\ast,\re} =
\| \mA \|_{\ast, \cpx} = \| \mA \|_{2\ast,\cpx} = 2/3$.
The real nuclear decomposition is given as
%
%
\[
\baray{rl}
\mA(x)\,\,  = & \frac{1}{192}
\Big( (-x_1-x_2+x_3+x_4)^4 +  (-x_1+x_2-x_3+x_4)^4 + \\
& \qquad \,(-x_1+x_2+x_3-x_4)^4 +  (+x_1+x_2+x_3+x_4)^4 - \\
& \qquad \,(-x_1+x_2+x_3+x_4)^4 -(x_1-x_2+x_3+x_4)^4 - \\
& \qquad \,(x_1+x_2-x_3+x_4)^4 -(x_1+x_2+x_3-x_4)^4 \, \Big).
\earay
\]
The above also serves as a complex nuclear decomposition.

\noindent
(iii) Consider the tensor $\mA \in \mt{S}^4(\re^2)$ such that
\[
\mA(x) = x_1^2 x_2^2.
\]
We got $\| \mA \|_{\ast, \re} = \| \mA \|_{2\ast, \re} = 1$.
The real nuclear decomposition is
\[
\mA(x) = \frac{1}{12} (-x_1+x_2)^4 + \frac{1}{12} (x_1+x_2)^4
-\frac{1}{6} x_1^4 -\frac{1}{6} x_2^4.
\]
The complex nuclear norm
$\| \mA \|_{\ast,\cpx} = \| \mA \|_{2\ast,\cpx} = 2/3$,
and the decomposition is
\[
\mA(x) = \frac{1}{24}
\Big( (x_1-x_2)^4 + (x_1+x_2)^4 - (x_1+\sqrt{-1}x_2)^4 - (x_1-\sqrt{-1}x_2)^4 \Big).
\]

\noindent
(iv) Consider the tensor $\mA \in \mt{S}^4(\re^3)$ such that
\[
\mA(x) = x_1^2 x_2^2 + x_2^2 x_3^2 + x_1^2 x_3^2.
\]
We got $\| \mA \|_{\ast, \re} = \| \mA \|_{2\ast, \re} = 2$.
The real nuclear decomposition is
\[
\baray{rl}
\mA(x) = & \frac{1}{24} \Big( (x_1-x_2+x_3)^4 +
(x_1+x_2-x_3)^4 +  (-x_1+x_2+x_3)^4 + \\
& \qquad + (x_1+ x_2+x_3)^4 \Big)
- \frac{1}{6} \Big( x_1^4 + x_2^4 + x_3^4 \Big).
\earay
\]
The complex nuclear norm
$\| \mA \|_{\ast,\cpx} = \| \mA \|_{2\ast,\cpx} = 5/3$,
and the decomposition is
\[
\baray{rl}
\mA(x) \,\, = & \frac{1}{36} \Big(
 (-x_1+x_2+x_3)^4+ (x_1-x_2+x_3)^4+ (x_1+x_2-x_3)^4+  (x_1+x_2+x_3)^4 \\
& \qquad  - (x_1-\sqrt{-1} x_3)^4 - (x_2-\sqrt{-1} x_3)^4 - (x_1-\sqrt{-1} x_2)^4\\
& \qquad  - (x_1+\sqrt{-1} x_3)^4 - (x_2+\sqrt{-1} x_3)^4 - (x_1+\sqrt{-1} x_2)^4
 \Big).
\earay
\]

\noindent
(v) Consider the tensor $\mA \in \mt{S}^4(\re^3)$ such that
\[
\mA(x) = (x_1^2 + x_2^2 + x_3^2)^2.
\]
We got $\| \mA \|_{\ast, \re} = \| \mA \|_{2\ast, \re} =
\| \mA \|_{\ast, \cpx} = \| \mA \|_{2\ast, \cpx} = 5$,
with the nuclear decomposition
\[
\baray{rl}
\mA(x) = & \frac{1}{12} \Big( (x_1-x_2+x_3)^4 +
(x_1+x_2-x_3)^4 +  (-x_1+x_2+x_3)^4 + \\
& \qquad + (x_1+ x_2+x_3)^4 \Big)  + \frac{2}{3} \Big( x_1^4 + x_2^4 + x_3^4 \Big).
\earay
\]
The above also serves as a complex nuclear decomposition.
\qed
\end{exm}

\begin{exm} \label{sym:abc}
For $a,b,c \in \re^n$, the symmetrization of the
rank-$1$ nonsymmetric tensor $a\otimes b \otimes c$ is
\[
\baray{r}
\mbf{sym}(a \otimes b \otimes c) := \frac{1}{6}
(a \otimes b \otimes c +  a \otimes c \otimes b
+ b \otimes a \otimes c \quad \, \\
  + b \otimes c \otimes a +
c \otimes a \otimes b + c \otimes b \otimes a \, ).
\earay
\]
One wonders whether $\| \mbf{sym}(a \otimes b \otimes c) \|_{\ast, \re}
= \|a\| \cdot \|b\| \cdot \|c\|$ or not.
Indeed, this is usually not true. Typically, we have the inequalities
\[
\| \mbf{sym}(a \otimes b \otimes c) \|_{\ast, \cpx}
< \| \mbf{sym}(a \otimes b \otimes c) \|_{\ast, \re}
< \|a\|  \cdot \|b\| \cdot \|c\|.
\]
For instance, consider the following tensor in $\mt{S}^3(\re^3)$
\[
\mA = \mbf{sym}( e_1 \otimes (e_1+e_2) \otimes (e_1+e_2+e_3) ).
\]
We can compute that
$
\| \mA \|_{\ast, \cpx} \approx 2.2276,
$
$
\| \mA \|_{\ast, \re} \approx  2.4190,
$
but
\[
\| e_1 \| \cdot \| e_1+e_2 \| \cdot \| e_1+e_2+e_3 \| = \sqrt{6} \approx  2.4495.
\]
The computed real nuclear decomposition
$\mA = \sum_{i=1}^5 (u_i)^{\otimes 3}$ is {\tiny
\begin{verbatim}
   -0.2896    0.0149   -0.4750   -0.0947    1.0423
    0.1803   -0.5617    0.1042   -0.2944    0.5806
   -0.2891   -0.3132    0.3114    0.1865    0.2630
\end{verbatim} \noindent}The
computed complex nuclear decomposition
$\mA = \sum_{i=1}^8 (w_i)^{\otimes 3}$ is {\tiny
\begin{verbatim}
  -0.2795 + 0.2861i  -0.2882 + 0.2378i  -0.4808 + 0.8488i  -0.2489 + 0.0952i
   0.1094 + 0.2837i  -0.0298 + 0.2075i  -0.2776 + 0.4375i   0.1376 + 0.3928i
  -0.2009 + 0.0696i   0.2301 + 0.1457i  -0.1146 + 0.2252i   0.1359 + 0.1297i

   0.1212 + 0.0018i   0.3286 - 0.1764i   0.2244 - 0.1049i   0.2196 - 0.1357i
   0.0304 - 0.0450i   0.2410 + 0.1640i   0.1103 + 0.2279i   0.2235 + 0.2601i
   0.0268 + 0.0974i  -0.0096 + 0.2903i   0.1464 - 0.1210i   0.0750 + 0.1106i
\end{verbatim}}
\noindent
However, if $a=b=c$, then
$\| \mbf{sym}(a\otimes b\otimes c) \|_{\ast, \F} = \| a \|^3$
for $\F = \re, \cpx$.
\qed
\end{exm}

\begin{exm}
(Sum of Even Powers)
For an even order $m$, consider the tensors $\mA \in \mt{S}^m(\re^n)$
of the form such that
\be \label{SOEP}
\mA = (a_1)^{\otimes m} + \cdots + (a_r)^{\otimes m},
\ee
with $a_1, \ldots, a_r \in \re^n$.
Such a tensor is called a sum of even powers (SOEP).
Interestingly, for all SOEP tensors as in \reff{SOEP}, we have
\be \label{tnn:spow}
\| \mA \|_{\ast, \re} = \| \mA \|_{\ast, \cpx}  =  \| a_1 \|^{m} + \cdots + \| a_r \|^{m}.
\ee
Clearly,
$\| \mA \|_{\ast, \cpx} \leq  \| \mA \|_{\ast, \re}
\leq \sum_{i=1}^r \| a_i \|^{m}.$
The reverse inequalities are actually also true.
Let $\mB$ be the tensor such that $\mB(x) = (x^Tx)^{m/2}$.
Then, $\| \mB \|_{\sig, \cpx} = 1$. By the duality relation, we get
\[
\| \mA \|_{\ast, \cpx} \geq
\mA \bullet \mB = \sum_{i=1}^r \mB(a_i)  =
\sum_{i=1}^r \| a_i \|^m.
\]
So, \reff{tnn:spow} is true. It can also be proved by
applying Lemma~4.1 of \cite{FriLim14b}.
Moreover, every real nuclear decomposition of an SOEP tensor
must also be in the SOEP form.
This can be shown as follows. Suppose $\mA = \mA_1 - \mA_2$
is a real nuclear decomposition, with $\mA_1, \mA_2$ being SOEP tensors
such that $\| \mA \|_{\ast, \re} =
\| \mA_1 \|_{\ast, \re} + \| \mA_2 \|_{\ast, \re}$.
Then, $\mA_1 = \mA + \mA_2$ and
$
\| \mA_1 \|_{\ast, \re} =  \| \mA \|_{\ast, \re} + \| \mA_2 \|_{\ast, \re}
$
by \reff{tnn:spow}.
So, we must have $\| \mA_2 \|_{\ast, \re}=0$, hence $\mA_2=0$.
This shows the real nuclear decomposition is also SOEP.
For instance, consider the following SOEP tensor $\mA \in \mt{S}^4(\re^3)$
\[
\mA = \sum_{i=1}^3 \left( (e + e_i)^{\otimes 4} + (e - e_i)^{\otimes 4} \right).
\]
Algorithms~\ref{alg:even:m} and \ref{alg:cnn:cpx}
confirmed that $\| \mA \|_{\ast, \re} = \| \mA \|_{\ast, \cpx} = 120$.
\qed
\end{exm}

\section{Extensions to nonsymmetric tensors}
\label{sc:exten}

The methods in this paper can be naturally extended to nonsymmetric tensors.
A similar discussion was made in \cite{TanSha15}.
For convenience, we show how to do this for a nonsymmetric cubic tensor
$\mA \in \re^{n_1 \times n_2 \times n_3}$.
Clearly, its real nuclear norm can be computed as
\be \label{nsymA:nn3}
\| \mA \|_{\ast,\re} = \min \left\{ \sum_{i=1}^r \lmd_i
\left| \baray{c}
\mA = \Sig_{i=1}^r  \lmd_i v^{(i,1)} \otimes v^{(i,2)} \otimes v^{(i,3)}, \\
\lmd_i \geq 0, \, \| v^{(i,j)} \| = 1, \, v^{(i,j)} \in \re^{n_j}
\earay\right.
\right \}.
\ee
One can similarly show that
$\| \mA \|_{\ast,\re}$ is equal to the minimum value
of the optimization problem
\be  \label{mOp:nsyA:m=3}
\left\{\baray{rl}
 \min  &  \int 1 \mt{d} \mu  \\
s.t. &  \mA = \int  x^{(1)} \otimes x^{(2)} \otimes x^{(3)} \mt{d} \mu, \,\,
  \mu \in \mathscr{B}(T).
\earay \right.
\ee
In the above, the variables  $x^{(j)} \in \re^{n_j}$,
and $\mathscr{B}(T)$ is the set of Borel measures supported on the set
\[
T := \big\{ (x^{(1)}, x^{(2)}, x^{(3)}):
\| x^{(1)} \| = \| x^{(2)} \| = \| x^{(3)} \| = 1 \big  \}.
\]
Similarly, we can define the cone of moments
(denote $[n] :=\{1,\ldots,n\}$)
\be
\mathscr{R}_{ \{0, 3\} }^{n_1,n_2,n_3} := \left\{
y \in \re^{ 1 + n_1n_2n_3 }
\left| \baray{c}
 \exists \mu \in \mathscr{B}(T) \quad s.t. \\
\,  (y)_{ijk} = \int (x^{(1)})_i  (x^{(2)})_j (x^{(3)})_k  \mt{d} \mu  \\
i \in [n_1],\, j \in [n_2],\, k \in [n_3],  \\
\quad \mbox{ or } \quad  i=j=k=0
\earay \right.
\right\}.
\ee
One can show that \reff{mOp:nsyA:m=3} is equivalent to
\be \label{yOpt:ns:m=3}
\left\{\baray{rl}
\min  &  (y)_{000}  \\
s.t. &  (y)_{ijk} =  \mA_{ijk} \, \, (\forall i,j,k), \\
 &  y \in \mathscr{R}_{ \{0, 3\} }^{n_1,n_2,n_3}.
\earay \right.
\ee
A similar version of Algorithm~\ref{alg:R:odd} can be
applied to solve \reff{yOpt:ns:m=3}.
%
%

\begin{exm}
Consider the nonsymmetric tensor
$\mA \in \re^{2 \times 2 \times 2}$ such that
\[
\mA_{ijk} = i -j -k.
\]
By solving \reff{yOpt:ns:m=3},
we get $\| \mA \|_{\ast, \re} = 6.0000$.
A real nuclear decomposition of $\mA$ is given as {\tiny
\[
\left[ \baray{r}  1.4363  \\  0.3140   \earay \right] \otimes
\left[ \baray{r}   -0.9146 \\  -1.1516   \earay \right] \otimes
\left[ \baray{r}   0.9296  \\  1.1394 \earay \right] +
\left[ \baray{r}   0.7375 \\   1.0525    \earay \right] \otimes
\left[ \baray{r}     -0.2430 \\  -1.2616     \earay \right] \otimes
\left[ \baray{r}    0.5250 \\   1.1727  \earay \right]
\]
\[
+\left[ \baray{r}  0.5484  \\  0.6978   \earay \right] \otimes
\left[ \baray{r}   0.8846  \\  0.0732    \earay \right] \otimes
\left[ \baray{r}  0.6501 \\  -0.6040   \earay \right].
\]
}
\end{exm}

\bigskip
When $\F = \cpx$, we can similarly compute the
complex nuclear norm $\| \mA \|_{\ast,\cpx}$,
by considering each $x^{(j)}$ as a complex variable.
For nonsymmetric tensors,
it is usually much harder to compute the nuclear norm
$\| \mA \|_{\ast,\re}$ or $\| \mA \|_{\ast,\cpx}$.
This is because the variable $x$ has much higher dimension
than for the case of symmetric tensors,
which makes the moment optimization problem like
\reff{yOpt:ns:m=3} very difficult to solve.

\bigskip
\noindent
{\bf Acknowledgement}
The author would like to thank Shmuel Friedland
and Lek-Heng Lim for comments on tensor nuclear norms.
The research was partially supported by the NSF grant
DMS-1417985.


\begin{thebibliography}{100}


\bibitem{banach}
{\sc S.~Banach},
{\em \"{U}ber homogene polynome in $(L^2)$},
Studia Math., 7 (1938), pp. 36-44.



\bibitem{BTN}
{\sc A.~Ben-Tal and A.~Nemirovski},
{\em Lectures on Modern Convex Optimization: Analysis, Algorithms,
and Engineering Applications}.
MPS-SIAM Series on Optimization, SIAM, Philadelphia, 2001.



%
%
%
%


\bibitem{BVbook}
{\sc S.~Boyd and L.~Vandenberghe},
{\em Convex Optimization},
Cambridge University Press, 2004.



%
%
%



%
%

%
%

\bibitem{CurtoF} {\sc R. Curto and L. Fialkow},
{\em Truncated K-moment problems in several variables},
J. Operator Theory, 54 (2005), pp. 189--226.



\bibitem{Der13}
{\sc H.~Derksen},
{\em On the nuclear norm and the singular value decomposition of tensors},
Found. Comput. Math.,
\, 16 (2016), no. 3, 779-811.

%
%



\bibitem{Fialkow}
{\sc L. Fialkow and J. Nie},
{\em The truncated moment problem via homogenization and flat extensions},
J. Funct. Anal., 263 (2012), pp. 1682--1700.



\bibitem{Fri13}
 {\sc S. Friedland},
{\em Best rank one approximation of real symmetric tensors can be chosen symmetric},
Front. Math. China 8 (2013), no. 1, 19-40.


\bibitem{FriOtt13}
{\sc S.~Friedland and G.~Ottaviani},
{\em The number of singular vector tuples and uniqueness
of best rank one approximation of tensors},
Found. Comput. Math.,
Vol.~14, No.~6,  pp. 1209-1242, 2014.


%
%
%


\bibitem{FriLim14b}
{\sc S.~Friedland and L.-H.~Lim},
{\em Nuclear norm of higher-order tensors},
Preprint, 2016.
\url{https://www.stat.uchicago.edu/~lekheng/work/nuclear.pdf}



\bibitem{FriLim16a}
{\sc S.~Friedland and L.-H.~Lim},
{\em The computational complexity of duality},
Preprint, 2016.
\url{arXiv:1601.07629 [math.OC]}



%
%


%
%
%
%

\bibitem{Helton}
{\sc J. W. Helton and J. Nie}, {\em A semidefinite approach for truncated
K-moment problems}, Found. Comput. Math.,  12 (2012), pp. 851--881.



%
%
%


\bibitem{Gloptipoly}
{\sc D. Henrion, J. Lasserre  and J. Loefberg},
{\em GloptiPoly 3: moments, optimization and semidefinite programming},
 Optim. Methods Softw., 24 (2009), pp. 761--779.


\bibitem{HL13}
{\sc C.~Hillar and L.-H.~Lim},
{\em Most tensor problems are NP-hard},
Journal of the ACM, 60 (2013), no. 6, Art. 45, 39 pp.


\bibitem{Land12}
{\sc J.M.~Landsberg},
{\em Tensors: geometry and applications},
Graduate Studies in Mathematics, 128.
American Mathematical Society, Providence, RI, 2012.




\bibitem{Lim13}
{\sc L.-H. Lim},
{\em Tensors and hypermatrices}, in: L. Hogben (Ed.),
{\em Handbook of linear algebra}, 2nd Ed.,
CRC Press, Boca Raton, FL, 2013.




\bibitem{Lasserre01}
{\sc  J. B. Lasserre},
{\em Global optimization with polynomials and the problem of moments},
SIAM J. Optim., 11 (2001),  pp. 796--817.



\bibitem{Lasserre09}
{\sc J. B. Lasserre},
{\em Moments, Positive Polynomials and Their Applications}, Imperial College
Press, 2009.

\bibitem{Lasserre15}
{\sc J.B. Lasserre},
{\em Introduction to Polynomial and Semi-Algebraic Optimization},
Cambridge University Press, Cambridge, 2015.



\bibitem{Lasserre08} {\sc  J.B. Lasserre},
{\em  A semidefinite programming approach to the generalized problem of moments},
Math. Program., 112 (2008), pp. 65--92.


%
%

\bibitem{Laurent}
{\sc  M. Laurent},
{\em Sums of squares, moment matrices and optimization over polynomials},
Emerging Applications of Algebraic Geometry, Vol. 149 of IMA Volumes in Mathematics and its
Applications, M. Putinar and S. Sullivant (eds), Springer, 2009, pp. 157--270.


%
%



\bibitem{LimCom10}
{\sc L.-H.~Lim and P.~Comon},
{\em Multiarray signal processing: tensor decomposition meets compressed sensing},
Comptes Rendus de l'Académie des sciences, Series IIB – Mechanics,
338 (2010), no. 6, pp. 311–320.


\bibitem{LimCom14}
{\sc L.-H.~Lim and P.~Comon},
{\em Blind multilinear identification},
IEEE Transactions on Information Theory, 60 (2014), no. 2, pp. 1260–1280.



\bibitem{Marsh03}
{\sc M.~Marshall},
{\em Optimization of polynomial functions},
Canadian Math. Bull., \, 46~(4), pp. 575-587, 2003.



\bibitem{MHWG}
{\sc C.~Mu, B.~Huang, J.~Wright, and D.~Goldfarb},
{\em Square deal: lower bounds and improved relaxations for tensor recovery},
Preprint, 2013.
\url{arXiv:1307.5870 [stat.ML]}



%
%

%
%

%
%

%
%



\bibitem{ATKMP}
{\sc J. Nie},
{\em The $\mathcal{A}$-truncated K-moment problem}, Found. Comput. Math., 14 (2014), pp. 1243--1276.


\bibitem{LMOPT}
{\sc J. Nie},
{\em Linear optimization with cones of moments and nonnegative polynomials},
Math. Program., Ser. B, 153 (2015), pp. 247--274.


\bibitem{opcd}
{\sc J. Nie},
{\em Optimality conditions and finite convergence of Lasserre's hierarchy},
Math. Program., Ser. A,  146 (2014), pp. 97--121.




\bibitem{GPSTD}
{\sc J.~Nie},
{\em Generating polynomials and symmetric tensor decompositions},
Found. Comput. Math., to appear.




\bibitem{NW14}
{\sc J. Nie and L. Wang},
{\em Semidefinite relaxations for best rank-1 tensor approximations},
SIAM Journal on Matrix Analysis and Applications,
35 (2014), pp. 1155--1179.




\bibitem{OedOtt13}
{\sc L.~Oeding and G.~Ottaviani},
{\em Eigenvectors of tensors and algorithms for Waring decomposition},
J. Symbolic Comput.,  54, 9-35, 2013.



\bibitem{Putinar} {\sc M. Putinar},
{\em Positive polynomials on compact semi-algebraic sets},
Ind. Univ. Math. J.,
42 (1993), pp. 969--984.


\bibitem{sedumi}
{\sc J. F. Sturm},
{\em SeDuMi 1.02: A MATLAB toolbox for optimization over symmetric cones},
Optim. Methods Softw., 11 \& 12 (1999), pp. 625--653.


\bibitem{TanSha15}
{\sc G.~Tang and P.~Shah},
{\em Guaranteed tensor decomposition: a moment approach},
Proceedings of the 32nd International Conference on Machine Learning (ICML-15),
pp. 1491-1500, 2015.



\bibitem{YuaZha15}
{\sc M.~Yuan and C.-H.~Zhang},
{\em On tensor completion via nuclear norm minimization},
Found. Comput. Math., to appear.



\bibitem{ZLQ12}
{\sc X.~Zhang, C.~Ling and L.~Qi.},
{\em The best rank-1 approximation of a symmetric tensor and related
  spherical optimization problems},
\newblock SIAM Journal on Matrix Analysis and Applications,
  33(3):806--821, 2012.








\end{thebibliography}
\end{document}